\documentclass[11pt,reqno,tbtags]{article}
\input{preamble.sty}

\title{Rare event probabilities in Random Geometric Graphs}
\author[1]{Prabhanka Deka}
\author[2]{Fangzhou Luo}
\author[2]{Baichuan Wu}
\affil[1]{Beijing International Center for Mathematical Research, Peking University}
\affil[2]{School of Mathematical Sciences, Peking University}
\date{}


\begin{document}

\maketitle

\begin{abstract}
In this paper, we study rare events in spherical and Gaussian random geometric graphs in high dimensions. In these models, the vertices correspond to points sampled uniformly at random on the $d$ dimensional unit sphere or correspond to $d$ dimensional standard Gaussian vectors, and edges are added between two vertices if the inner-product between their corresponding points are greater than a threshold $\uthreshold{p}$, chosen such that the probability of having an edge is equal to $p$. We focus on two problems: (a) the probability that the RGG is a complete graph, and (b) the probability of observing an atypically large number of edges. We obtain asymptotically exponential decay rates depending on $n$ and $d$ of the probabilities of these rare events through a combination of geometric and probabilistic arguments.
\end{abstract}

\section{Introduction}\label{S.Introduction}

Random graphs and networks have become fundamental tools for modeling complex systems across a wide range of disciplines, including computer science, biology, sociology, and physics. These models help capture the intricate connectivity patterns observed in real-world networks, such as the internet, social interactions, and neural connections. The study of random graphs provides insights into the structural properties, robustness, and dynamical behaviors of large-scale networks, making them indispensable in modern data science and engineering applications. The mathematical study of random graphs traces its origins to the seminal work of Erd\H{o}s and R\'enyi, who introduced a simple yet powerful model where edges are included independently with a fixed probability~\cite{erdos1960evolution}. Since then, many versions of this simple model have been introduced and studied, like the inhomogeneous random graph~\cite{bollobas2007phase}, the Chung-Lu model~\cite{aiello2000random}, and the configuration model~\cite{bollobas1980probabilistic}, to name a few. Further, dynamic models such as preferential attachment networks~\cite{barabasi1999emergence}, where new vertices arrive one at a time and connect to the existing network, were developed because most real world networks are growing in size over time, and explain the scale-free degree distributions observed in many real-world networks. Another important class of random graphs, the random geometric graphs (RGGs), introduced in Gilbert~\cite{gilbert1961random}, incorporate spatial structure by embedding vertices in a metric space and connecting them based on geometric proximity. One of the main difference between the RGG and the other models mentioned is the nature of dependence between the edges. In the Erd\H{o}s-R\'enyi model (and most other models mentioned above), all edges are independent. Meanwhile in the RGG, due to the latent geometry, if edges $(i,j)$ and $(j,k)$ are present in the graph, it is more likely that the edge $(i,k)$ is also present, leading to a more intricate dependence between the edges. This phenomenon, often referred to as clustering, is captured by the clustering coefficient of the network.

In this work, we study two related random geometric graph models, where $n$ vertices are sampled uniformly at random on the $d$ dimensional unit sphere or the $n$ vertices correspond to $d$ dimensional standard Gaussian vectors. Edges are formed between pairs of vertices whose inner product exceeds a threshold $t_p$, chosen such that each edge appears with probability $p$. Following terminology in~\cite{bangachev2024fourier}, we call the two models the spherical random geometric graph (SRGG) or the Gaussian random geometric graph (GRGG). These models differ from the more commonly studied RGG variant of Gilbert~\cite{gilbert1961random}, where vertices are generated via a Poisson point process of intensity $\lambda$ in $\bbR^d$ and two points are connected if their Euclidean distance is below a threshold radius $r$. The spherical RGG model is particularly relevant in high dimensional statistics and machine learning, where inner-product-based similarity measures are widely used, such as in kernel methods and neural networks.

Although the model is fairly easy to describe, the dependence between the edges introduces major challenges in understanding and analyzing even very simple questions. In this paper, we investigate rare events in this model, focusing on two key problems: (a) the probability that the RGG is a complete graph, and (b) the probability of observing an atypically large number of edges. These questions are motivated by applications in outlier detection, hypothesis testing, and understanding the extremal behavior of geometric random graphs. For example, the complete graph probability is linked to the study of clique formations in high dimensional data, while large deviation results for the number of edges shed light on the robustness of geometric network models. 

Previous works \cite{seppalainen2001large,schreiber2005large,chatterjee2020localization,hirsch2020lower,hirsch2023large} studied large deviation principles of certain geometric functionals in various models, where they fixed $d$ and let $n$ go to infinity. We focus our attention on the dense RGG in high dimensions, where the edge probability $p>0$ is fixed, and we let both $n$ and $d$ go to infinity. For the first problem, let $\uclique$ be the event that the SRGG on $n$ vertices is a complete graph. Table~\ref{tab:complete graph} summarizes our results for various regimes of $n$ and $d$, where the upper and lower bound columns list the dependence of $\log \bbP(\uclique)$ on $n$ and $d$. The constants $c$ and $C$ appearing in each row are possibly different, but do not depend on $n$ and $d$. In the regimes $d \ge n^2$ and $d \le n^{4/3 + o(1)}$, we obtain matching upper and lower bounds. In the intermediate regime, there is a discrepancy between the upper and lower bounds, which our proof technique was unable to address.

\begin{table}[ht]
    \centering
    \label{tab:complete graph}
    \begin{tabular}{||r@{\,}c@{\,}c@{\,}c@{\,}l|c|c||}
        \hline
        \multicolumn{5}{||c|}{Regime} & Lower bound & Upper bound\\
        \hline\hline
        $n^2$ & $\lesssim$ & $d$ & \  & \  & $- Cn^2$ & $- cn^2$\\
        \hline
        $n^2 / \log n$ & $\lesssim$ & $d$ & $\lesssim$ & $n^2$ & $- Cn \sqrt{d}$ &
        $- cn^2$\\
        \hline
        $n^{4 / 3 + o (1)}$ & $\lesssim$ & $d$ & $\lesssim$ & $n^2 / \log n$ & $-
        Cn \sqrt{d}$ & $- cn \sqrt{d \log n}$\\
        \hline
        $\log n$ & $\lesssim$ & $d$ & $\lesssim$ & $n^{4 / 3 + o (1)}$ & $- Cn
        \sqrt{d \log n}$ & $- cn \sqrt{d \log n}$\\
        \hline
        \  & \  & $d$ & $\lesssim$ & $\log n$ & $- Cnd$ & $- cnd$\\
    \hline
    \end{tabular} 
    \caption{Lower and upper bounds for $\log \bbP(\uclique)$.}
\end{table}

For the upper tail large deviation of the number of edges, we obtain matching upper and lower bounds in all regimes. We consider the event that the number of edges present in an RGG exceeds $(1+\deviationrate)$ times the expected number of edges. The probability of this event decays like $\exp(-C_1n^2)$ when $d\gtrsim \sqrt n$ and $\exp(-C_2n\sqrt d)$ when $d\lesssim \sqrt n$. Here $C_1$ and $C_2$ are constants that depend on $p$ and $\varepsilon$, but are independent of $n$ and $d$. In very high dimensions, the decay is of order $\exp(-\Omega(n^2))$, similar to the Erd\H{o}s-R\'enyi model, while in low dimensions, the two models behave differently.

\subsection{Related works}
There has been increasing interest in studying and understanding RGGs, especially in high dimensions. One of the first works that considered RGGs in high dimensions is~\cite{devroye2011high}, where the authors study the clique number in the regime $d \gg \log(n)$. An important question in high dimensions is understanding when the geometry of the underlying space is preserved. For dense RGGs ($p=\Theta(1)$), \cite{bubeck2016testing} established a sharp phase transition: geometry is detectable when the latent dimension $d\ll n^3$, but when $d\gg n^3$ this latent geometry is lost and the RGG becomes statistically indistinguishable from the Erd\H{o}s-R\'enyi model. For sparse RGGs ($p=O(1/n)$), \cite{bubeck2016testing} presented a polynomial time test that can distinguish between RGG and Erd\H{o}s-R\'enyi models when $d\ll (\log n)^3$, and conjectured that geometry is lost once $d\gg (\log n)^3$, and there has been a lot of progress towards proving this in recent years (see~\cite{avrachenkov2020cliques,brennan2020phase,liu2022testing}).

\cite{chatterjee2020localization} studied the large deviation principles for the number of edges in the RGG model where the vertices are generated by a Poisson point process of intensity $n$ on the torus $\mathbb T^d$, and edges are added between vertices whose Euclidean distance is less than a threshold $r_n$. The distance threshold is chosen such that the expected number of edges grows as a $n^{\alpha}$ for some $\alpha \in (0,2)$. Their results show that conditioned on the graph having an atypically large number of edges, there exists a small ball that contains a large clique, and they further obtain a large deviations principle for the upper tail of the number of edges.

A recent work \cite{ma2025exponentialimprovementramseylower} considered an analogue of the RGG called \emph{random sphere graph} to obtain a new lower bound on off-diagonal Ramsey numbers, which is the first exponential improvement over the classical lower bound obtained by Erd\H{o}s in 1947. The proof is based on estimating the probability of the existence of smaller cliques in the random sphere graph (see~\cite{ma2025exponentialimprovementramseylower}, Theorem 3.1).

\subsection{Organization of the paper} 
The rest of the paper is organized as follows. We first introduce important definitions and notations in Section~\ref{S.Definitions}. We state our main results in Section~\ref{S.Main Results}, and present an outline of our proof along with a few necessary tools in~\Cref{S.Proof Outline and Tools}. Finally, in~\Cref{S.Estimating Clique Probability} and~\Cref{S.Estimating Edge Deviation Probability}, we prove the theorems stated in Section~\ref{S.Main Results}.

\subsection{Definitions and Notations} \label{S.Definitions}

In this section, we introduce the models that we consider in the paper along with some necessary definitions and notation. A graph $G=(V,E)$ is a pair, where $V$ is the set of vertices and $E \subseteq V \times V$ is the set of edges. The graphs we consider are simple, undirected graphs, i.e., the collection of edges consists of unordered pairs without repetitions and self loops. A random graph $\cG_n$ is a probability measure on the space of all simple graphs with $n$ vertices. 

\paragraph{\textbf{Random geometric graphs.}} For some $p \in [0,1]$, we consider the following two random geometric graph models.
\begin{enumerate}
    \item \textbf{Spherical model.} Let $(\unif_1,\ldots,\unif_n)$ be \iid \ random variables having distribution Unif$(\sphere)$. Let $\uthreshold{p}$ be such that $\bbP(\inner{\unif_1}{\unif_2} \geq \uthreshold{p}) = p$. Define the spherical random geometric graph $\srgg(n,p)$ to be the random graph on vertices $V(\srgg(n,p)) = [n]$ and edges $E(\srgg(n,p))$ given by $(i,j) \in E(\srgg(n,p))$ if and only if $\inner{\unif_i}{\unif_j} \geq \uthreshold{p}.$ 
    \item \textbf{Gaussian model.} Let $(\gauss_1,\ldots,\gauss_n)$ be \iid \ random variables having distribution $\cN(0,I_d)$, the standard normal distribution on $\bbR^d$. Let $\gthreshold{p}$ be such that $\bbP(\inner{\gauss_1}{\gauss_2} \geq \gthreshold{p}) = p$. Define the Gaussian random geometric graph $\grgg(n,p)$ to be the random graph on vertices $V(\grgg(n,p)) = [n]$ and edges $E(\grgg(n,p))$ given by $(i,j) \in E(\grgg(n,p))$ if and only if $\inner{\gauss_i}{\gauss_j} \geq \gthreshold{p}.$ 
\end{enumerate}
Note that the thresholds $\uthreshold{p}$ and $\gthreshold{p}$ also depend on the dimension, $d$, of the underlying space, but we suppress that dependence in our notation. We write $\uthreshold{p,d}$ and $\gthreshold{p,d}$ when the dependence on dimension needs to be made explicit. One can also think of these models as random graphs in the corresponding space ($\sphere$ or $\bbR^d$, resp.) by treating vertex $i$ to be located at $\unif_i$ or $\gauss_i$.
Note that in both models, due to the choice of \emph{thresholds} $\uthreshold{p}$ and $\gthreshold{p}$, the probability of having an edge between two vertices $i$ and $j$ is given by 
\[
\bbP((i,j) \in E(\cG_{.}(n,p))) = p.
\]
Throughout this paper, we will work with RGGs in high dimensions, i.e., $d=d(n)$ grows with $n$.

\begin{remark} \label{R.Correspondence}
    For the special case $p=1/2$, where we have $\uthreshold{1/2}=\gthreshold{1/2}=0$, and the two models induce the same random graph measure due to the spherical symmetry of the $d$ dimensional standard normal distribution. 
\end{remark}

\paragraph{\textbf{Measures on $\sphere$ and $\bbR^d$}.} We write $\umeasure{d-1}$ for the uniform probability measure on $\sphere$ and $\gmeasurearg{d}$ for the standard $d$ dimensional Gaussian probability measure on $\bbR^d$. When it is clear which case we are working on, we drop the subscript. Similarly, when the dimension is clear from context, we drop the superscript.

\paragraph{\textbf{Edge-weighted graphs.}} An \emph{edge-weighted graph} $\ewgraph(V)$ is a graph on vertex set $V$, where each unordered pair $(v,w)$ with $v,w \in V$ is assigned an $\bbR$-valued weight $\eweight_{vw}^H$. We will omit the superscript $H$ when the graph is clear from context. For two edge-weighted graphs $H(V)$ and $H'(V)$ on the same vertex set $V$, we write $H \leq H'$ if $\eweight_{vw}^H \leq \eweight_{vw}^{H'}$ for all $v,w \in V.$
We list a few special cases of the edge-weighted graphs that we use. 
\begin{enumerate}
    \item The \emph{constant} weighted graph with all edge-weights equal to $c$, which we denote by $c\1_n$. 
    \item The \emph{edge-weighted geometric graph} $\wgg{U}{n}$, where for each $v \in V,$ we associate a point $x_v \in U \subseteq \bbR^d$, and for each unordered pair $(v,w)$, we assign a weight $\eweight_{vw} = \inner{x_v}{x_w}$. 
\end{enumerate}
For our results, we are interested in the special cases of the \emph{weighted geometric graph} with $U = \sphere$ (resp., $U = \bbR^d$), and the points associated to the vertices are chosen independently on $U$ with distribution $\umeasure{d-1}$ (resp., $\gmeasurearg{d}$). We call these \emph{edge-weighted random geometric graphs}, whose law will be denoted by $\rswgg{n}$ in the spherical case and $\rgwgg{n}$ in the Gaussian case.

\paragraph{\textbf{Asymptotic Notation.}} Throughout the paper, we will use Landau notation in the standard sense. Given two sequences $f_n$ and $g_n$ with $g_n > 0$, we say $f_n = O(g_n)$ or $f_n \lesssim g_n$ if there is a constant $M$ such that $|f_n| \le M g_n$ for all $n$ large enough, and say $f_n = o(g_n)$ if $\lim_{n \rightarrow \infty} f_n/g_n = 0$. We write $f_n\asymp g_n$ if $f_n = O(g_n)$ and $g_n = O(f_n)$.

\paragraph{\textbf{Indicator functions.}} Throughout the paper, we will use $\indicatorfn{A}(\cdot)$ as the indicator function of a subset $A$, and $\indicatorevent{\mathcal B}$ as the indicator of an event $\mathcal B$.

\section{Main Results}\label{S.Main Results}
In this section, we present our main results. Our first result is about the rare event that the spherical random geometric graph is a complete graph. We will always assume that $p\le1/2$ for technical reasons.

\begin{theorem} \label{T.SClique}
    Let $G \sim \srgg(n,p)$. Let $\uclique$ be the event that $G$ is a complete graph on $n$ vertices, i.e., 
    \begin{align}
    \uclique = \uclique(n,d,p) := \left\{ E(G) = ((i,j) : 1 \leq i \neq j \leq n) \right\} = \left \{ \inner{\unif_i}{\unif_j} \geq \uthreshold{p} \ \forall \ 1 \leq i \neq j \leq n \right \}. \label{D.UClique}
    \end{align}
    \begin{enumerate}
        \item \textbf{Lower bound.} There exist $C>0$ s.t. for all $n,d\ge 2$ 
        \begin{align}
            \bbP(\uclique) \geq \exp\left(-C \min(n^2, n \sqrt{d \log n}, nd) \right).
            \label{E.UCliqueLB}
        \end{align}
        \item \textbf{Upper bound.} There exist $C>0$ s.t. for all $n,d\ge 2$
        \begin{align}
            \bbP(\uclique) \leq \exp\left(-c\min\left(n^2,n\sqrt{d\log \left(\frac{n^{1-o(1)}}{d^{3/4}}\lor e \right)},nd\right)\right). \label{E.UCliqueUB}
        \end{align}
    \end{enumerate}
\end{theorem}

We have a similar result about the same event for the Gaussian random geometric graph.

\begin{theorem} \label{T.GClique}
    Let $G \sim \grgg(n,p)$. Let $\gclique$ be the event that $G$ is a complete graph on $n$ vertices, i.e.,
    \begin{align}
    \gclique = \gclique(n,d,p) := \left\{ E(G) = ((i,j) : 1 \leq i \neq j \leq n) \right\} = \left \{ \inner{\gauss_i}{\gauss_j} \geq \gthreshold{p} \ \forall \ 1 \leq i \neq j \leq n \right \}. \label{D.GClique}
    \end{align}
    \begin{enumerate}
        \item \textbf{Lower bound.} There exists $ C > 0$ s.t. for all $n,d\ge 2$,
        \begin{align}
            \bbP(\gclique) \geq \exp\left(-C \min(n^2, n \sqrt{d \log n}, nd) \right). \label{E.GCliqueLB}
        \end{align}
        \item \textbf{Upper bound.} There exists $ c > 0$ s.t. for all $n,d\ge 2$,
        \begin{align}
            \bbP(\gclique) \leq \exp\left(-c\min\left(n^2,n\sqrt{d\log \left(\frac{n^{1-o(1)}}{d^{3/4}}\lor e\right)}, nd\right)\right). \label{E.GCliqueUB}
        \end{align}
    \end{enumerate}
\end{theorem}

The next set of theorems present bounds on the probability that the RGGs have an atypically large number of edges. Recall that in both models, the expected number of edges is $\binom{n}{2}p$. For any fixed $\deviationrate > 0$ such that $(1+\deviationrate)p < 1$, we obtain exponential rates of decay for the probability of the event $\left\{E(\cG_{.}(n,p)) \ge (1+\deviationrate)\bbE\left[E(\cG_{.}(n,p))\right]\right\}$ as follows.
\begin{theorem}\label{T.SDeviation}
    Let $G \sim \srgg(n,p)$ and $\setsize{E(G)}$ be the number of edges in $G$. For any $\deviationrate > 0,$ define the event 
    \begin{align}
        \uedgedeviation{\deviationrate} := \left\{ \setsize{E(G)} \ge (1+\deviationrate) \binom{n}{2}p\right\}. \label{D.Uedgedeviation}
    \end{align}
    Then there exists constants $c, C > 0$ such that
    \begin{align}
        \exp\left(-c \min(n^2, n \sqrt{d}) \right) \le \bbP(\uedgedeviation{\deviationrate}) \le \exp\left(-C \min(n^2, n \sqrt{d}) \right). \label{E.Uedgedeviation}
    \end{align}
\end{theorem}

We also have a similar result for Gaussian random geometric graph.

\begin{theorem}\label{T.GDeviation}
    Let $G \sim \grgg(n,p)$ and $\setsize{E(G)}$ be the number of edges in $G$. For any $\deviationrate > 0,$ define the event 
    \begin{align}
        \gedgedeviation{\deviationrate} := \left\{ \setsize{E(G)} \ge (1+\deviationrate) \binom{n}{2}p\right\}. \label{D.Gedgedeviation}
    \end{align}
    Then there exists constants $c, C > 0$ such that
    \begin{align}
        \exp\left(-c \min(n^2, n \sqrt{d}) \right) \le \bbP(\gedgedeviation{\deviationrate}) \le \exp\left(-C \min(n^2, n \sqrt{d}) \right). \label{E.Gedgedeviation}
    \end{align}
\end{theorem}
\begin{remark}
    Both constants $c$ and $C$ above depend on the edge probability $p$ and the deviation rate $\deviationrate$. For a more explicit characterization of the constants, see Propositions~\ref{P.ERlowerbound2} and~\ref{P.biaslowerbound2} for the lower bounds, and Remark \ref{R.explicit constants} for the upper bounds.
\end{remark}

\section{Tools and Proof Outline}\label{S.Proof Outline and Tools}

\subsection{Coupling the two models}\label{S.Coupling}

As noted in Remark~\ref{R.Correspondence}, the Spherical RGG and the Gaussian RGG give the same random graph measure for the special case $p = 1/2.$ While the two models differ for $p \neq 1/2$, there is a natural way to couple the two models using the same set of Gaussian vectors using the observation that when $\gauss \sim \cN(0,I_d)$, the projection of $\gauss$ onto $\sphere$, $\unif = \gauss / \norm{\gauss}$, is distributed uniformly on $\sphere.$ In this section, we derive results that relate the probabilities of observing a clique in the two models, and similar results for the probability of having an atypically large number of edges. As a consequence, it will suffice to obtain upper and lower bounds for only one of the two models, combined with the results of this section, to complete the proofs of our theorems. The following lemma captures the asymptotic behaviour of the thresholds $\uthreshold{p}$ and $\gthreshold{p}$ as the dimension $d$ goes to infinity.
\begin{lemma}\label{L.estimate threshold}
    For fixed $p\in(0,1)$, we have 
    \begin{center}
        $\gthreshold{p}=(\gtailinverse{p}+o(1))\sqrt d$ \qquad and \qquad $\uthreshold{p}=(\gtailinverse{p}+o(1))/\sqrt d$
    \end{center}
    where $\Phi(\cdot)$ is the tail of the standard Gaussian distribution.
\end{lemma}

\begin{proof}
    Recall that $\gthreshold{p}$ is defined via $\bbP\left(\inner{\gauss_1}{\gauss_2} \ge \gthreshold{p}\right) = p,$ where $\gauss_1$ and $\gauss_2$ are two independent $\cN(0,I_d)$ random variables. Then we have that 
    \[
    \inner{\gauss_1}{\gauss_2} = \sum_{j=1}^d \gauss_1^{(j)} \gauss_2^{(j)},
    \]
    where $\gauss_i^{(j)}$, the $j$-th coordinate of $\gauss_i$ is a standard 1 dimensional Gaussian random variable. Note that the collection $\left\{ \gauss_{i}^{(j)} : i \in \{1,2\}, 1 \leq j \leq d \right\}$ is independent. Further, $\left(\gauss_1^{(j)} \gauss_2^{(j)}, 1\leq j \leq d \right)$ are independent and $\bbE\left[\gauss_1^{(j)} \gauss_2^{(j)}\right] = 0$, $\Var\left(\gauss_1^{(j)} \gauss_2^{(j)}\right) = 1$ for all $1 \le j \le d$. So $\Var (\inner{\gauss_1}{\gauss_2}) = d$. Using the Central Limit Theorem, we get that 
    \begin{align}
        \frac{\inner{\gauss_1}{\gauss_2}}{\sqrt{d}} \Rightarrow \cN(0,1) \label{E.CLTgaussianinner}
    \end{align}
    as $d \rightarrow \infty.$ Thus
    \[
    p = \bbP\left(\inner{\gauss_1}{\gauss_2} \ge \gthreshold{p}\right) = \bbP\left(\frac{\inner{\gauss_1}{\gauss_2}}{\sqrt{d}} \ge \frac{\gthreshold{p}}{\sqrt{d}} \right) \rightarrow \bbP\left(\cN(0,1) \ge \frac{\gthreshold{p}}{\sqrt{d}}\right) = \gtail{\frac{\gthreshold{p}}{\sqrt{d}}},
    \]
    which proves the first claim. To see the second part, we use the coupling representation and observe that 
    \[\
    \inner{\unif_1}{\unif_2} = \frac{\inner{\gauss_1}{\gauss_2}}{\norm{\gauss_1}_2 \norm{\gauss_2}_2}.
    \]
    By the Law of Large Numbers, $d^{-1/2}\norm{\gauss_i}_2 \xrightarrow{P} 1$ for $i = 1,2.$ Combining with the CLT in Equation~\eqref{E.CLTgaussianinner} and using Slutsky's Lemma, we get that 
    \begin{align}
        \sqrt{d} \inner{\unif_1}{\unif_2} \Rightarrow \cN(0,1) \label{E.CLTunifinner}
    \end{align}
    as $d \rightarrow \infty.$
    The second claim then follows from a similar argument.
\end{proof} 
Denote by $\couplegauss{n}{d}$ the law of $n$ independent $d$ dimensional standard Gaussian vectors. Let $\gaussvec = \left(\gauss_1,\ldots,\gauss_n \right) \sim \couplegauss{n}{d}$, and recall that $\unif_i := \gauss_i / \norm{\gauss_i}$ is distributed uniformly on $\sphere$. For any $1 \leq m \leq n$, define the events 
\begin{align*}
    \gcliquepar{m}{d}{p} &= \left\{ \inner{\gauss_i}{\gauss_j} \ge \gthreshold{p} \ \forall 1 \le i,j\le m \right\},\\
    \ucliquepar{m}{d}{p} &= \left\{ \inner{\unif_i}{\unif_j} \ge \uthreshold{p} \ \forall 1 \le i,j\le m \right\},
\end{align*}
where $\gthreshold{p}$ and $\uthreshold{p}$ are as defined earlier. The following lemma couples the two events.

\begin{lemma}\label{L.coupling}
    For fixed $p<1/2$ and any small $\delta>0$, there exist constants $c_\delta$,  $\Delta_\delta$, where $c_\delta$ and $\Delta_\delta$ both go to $0$ as $\delta$ goes to $0$ and $p<q<p+\Delta_\delta$ such that the following inequalities holds
    \begin{align}
        \couplegauss{n}{d}(\gcliquepar{n}{d}{p})\le \exp(-c_\delta dn)+2^n\couplegauss{n}{d}(\ucliquepar{(1-\delta)n}{d}{q}), \label{E.coupling1} \\
        \couplegauss{n}{d}(\ucliquepar{n}{d}{p})\le \exp(-c_\delta dn)+2^n\couplegauss{n}{d}(\gcliquepar{(1-\delta)n}{d}{q}). \label{E.coupling2}
    \end{align}
\end{lemma}
\begin{proof}
    We can write 
    \begin{align*}
        \gcliquepar{n}{d}{p} &= \left\{ \inner{\gauss_i}{\gauss_j} \ge \gthreshold{p} \ \forall 1 \le i,j\le n \right\}\\
        &= \left\{ \inner{\frac{\gauss_i}{\norm{\gauss_i}}}{\frac{\gauss_j}{\norm{\gauss_j}}} \ge \frac{\gthreshold{p}}{\norm{\gauss_i}\norm{\gauss_j}} \ \forall 1 \le i,j\le n \right\}\\
        &= \left\{ \inner{\unif_i}{\unif_j} \ge \uthreshold{p}\cdot\frac{\gthreshold{p}}{\uthreshold{p}}\cdot \frac{1}{\norm{\gauss_i}\norm{\gauss_j}} \ \forall 1 \le i,j\le n \right\}.
    \end{align*}
    By Lemma~\ref{L.estimate threshold}, the ratio $\gthreshold{p}/\uthreshold{p}\sim d$. Now, suppose for some subset $V \subseteq [n]$, we have that $\norm{\gauss_i} \le (1+\delta)\sqrt{d}$ for all $i \in V$. Assuming without loss of generality that $V = [m]$, where $m = |V|$, if the event 
    \begin{align*}
        \left\{ \inner{\gauss_i}{\gauss_j} \ge \gthreshold{p} \ \forall 1 \le i,j\le n \ \text{and} \  \norm{\gauss_i} \ge (1+\delta)\sqrt{d} \ \forall 1 \le i \le m  \right\} 
    \end{align*}
    holds, then we also have 
    \begin{align*}
        \left\{ \inner{\unif_i}{\unif_j} \ge t_q \ \forall \1 \leq i,j \le m \right\}
    \end{align*}
    holds for a new threshold $t_q = (1+\delta)^{-2} t_p(1 + o(1))$, where $p < q < p+\Delta_\delta$, which implies that
    \begin{align*}
        \couplegauss{n}{d} \left(\gcliquepar{n}{d}{p} \ \text{and} \ \norm{\gauss_i} \ge (1+\delta)\sqrt{d} \ \forall 1 \le i \le m \right) \le \couplegauss{n}{d}\left(\ucliquepar{m}{d}{q} \right).
    \end{align*}
    Thus, to obtain equation~\eqref{E.coupling1}, it suffices to show that we can take the subset $V$ above to be of size $\delta n$. To see that, first note that for a Gaussian vector $\gauss$, its length $\norm{\gauss}$ satisfies 
    $$\norm{\norm{\gauss}
    -\sqrt d}_{\psi_2}\le C,$$
    where $\norm{\cdot}_{\psi_2}$ is the sub-Gaussian norm (see Theorem 3.1.1 in~\cite{vershynin2018high}).
    Thus for a single Gaussian vector $\gauss$ and any $\delta > 0$, we have that $\bbP(\norm{\gauss}>(1+\delta)\sqrt d)\le\exp(-Cd)$ for some constant $C$ that depends only on $\delta$. A simple Bernoulli concentration gives that the probability of the event $\gnormdev{n}{\delta} := \left\{ \#\left\{i:\norm{\gauss_i}\ge(1+\delta)\sqrt d\right\}\ge\delta n\right\}$ is upper bounded by 
    $$\bbP\left(\gnormdev{n}{\delta}\right)\le\exp(-Cd\delta n) = \exp(-cdn),$$
    where $c >0$ is again a constant that depends only on $\delta.$
    Let $\mathcal J$ denote all the subsets of $[n]$ with size $\floor{(1-\delta)n}$. A union bound gives
    \begin{align*}
        \couplegauss{n}{d}(\gcliquepar{n}{d}{p}) &\le \couplegauss{n}{d} \left( \gcliquepar{n}{d}{p} \cap \gnormdev{n}{\delta} \right) + \couplegauss{n}{d}\left(\gcliquepar{n}{d}{p} \cap \gnormdev{n}{\delta}^c \right)\\
        &\le \exp(-cd n)+\sum_{I\in\mathcal J}\couplegauss{n}{d}\left(\gcliquepar{n}{d}{p} \ \text{and} \ \norm{\gauss_i}\le(1+\delta)\sqrt{d}~,\forall~i\in I\right)\\
        &\le \exp(-cdn)+2^n\cdot\bbP\left(\gcliquepar{n}{d}{p} \ \text{and} \ \norm{\gauss_i}\le(1+\delta)\sqrt{d}~,\forall~i \leq \floor{(1-\delta)n}\right)\\
        &\le \exp(-cdn)+2^n\cdot\couplegauss{n}{d}\left(\ucliquepar{(1-\delta)n}{d}{q} \right)
    \end{align*}
    as required. The proof of equation~\eqref{E.coupling2} follows from a similar argument.
\end{proof}

Similarly, we obtain the following lemma that couples the edge-deviation events in the two models. For any $1\le m\le n$, define the events
\begin{align*}
    \gedgedeviationpar{\deviationrate}{m}{d}{p}=\left\{\inner{\gauss_i}{\gauss_j}\ge\gthreshold{p}, ~for~at~ least~(1+\deviationrate)p\binom{m}{2}~pairs~of~1\le i<j\le m\right\},\\
    \uedgedeviationpar{\deviationrate}{m}{d}{p}=\left\{\inner{\unif_i}{\unif_j}\ge\uthreshold{p}, ~for~at~ least~(1+\deviationrate)p\binom{m}{2}~pairs~of~1\le i<j\le m\right\}.
\end{align*}
\begin{lemma}\label{deviation coupling}
    For fixed $p< 1/2$, $\deviationrate>0$, and any small $\delta > 0$, there exist constants $c_\delta$, $\Delta_\delta$, where $c_\delta$ and $\Delta_\delta$ both go to $0$ as $\delta$ goes to $0$, $p<q<p+\Delta_\delta$ and $0<\deviationrate'<\deviationrate$ such that the following inequality holds
    \begin{align}
        \bbP(\gedgedeviationpar{\deviationrate}{n}{d}{p})<\exp(-cdn)+2^n\bbP(\uedgedeviationpar{\deviationrate'}{(1-\delta)n}{d}{q}),  
        \label{deviation coupling1} \\
        \bbP(\uedgedeviationpar{\deviationrate}{n}{d}{p})<\exp(-cdn)+2^n\bbP(\gedgedeviationpar{\deviationrate'}{(1-\delta)n}{d}{q}).
        \label{deviation coupling2}
    \end{align}
\end{lemma}

The proof of this lemma is identical to the one for Lemma~\ref{L.coupling}, and we omit the details. Although Lemmas~\ref{L.coupling} and ~\ref{deviation coupling} are not tight, they allows us to estimate $\bbP(\uclique)$ and $\bbP(\gclique)$ (and similarly, $\bbP(\uedgedeviation{\deviationrate})$ and $\bbP(\gedgedeviation{\deviationrate})$) together, since the additive error term $\exp(-cdn)$ and the multiplicative factor $2^n$ are not comparable to the leading order terms for the clique probabilities.

\subsection{Geometric tools}\label{S.Geometric tools - Symmetrization}

In this subsection, we introduce \emph{symmetric rearrangement}, the core geometric tool used to establish the upper bounds in our theorems. While the definitions are borrowed from~\cite{draghici2005rearrangement} (see also~\cite{baernstein2019symmetrization} for a comprehensive treatment), we restate them here for the sake of completeness. The main idea is to use measure preserving transformations to transform arbitrary, measurable subsets of the sphere into spherical caps, defined below.

\paragraph{\textbf{Spherical caps.}} For a point $\bv \in \sphere$ and $a \in [0,1]$, a \emph{spherical cap} centered at $\bv$ with measure $a$, denoted by $\scap_{\bv}^a$, is the subset of $\sphere$ given by 
    \[
    \scap_{\bv}^{a} = \left\{\bw \in \sphere : \inner{\bw}{\bv} \geq h(a) \right\}.
    \]
where the \emph{altitude} of the spherical cap, $h(a) \in [-1,1]$, is chosen such that $\mu(\scap_{\bv}^a) = a.$ Recall that $\mu = \umeasure{d-1}$ is the uniform probability measure on $\sphere$. By definition, the cap centered at $\bv$ with measure $p$ is the neighborhood of a vertex located at $\bv$ in the SRGG $\srgg(n,p)$. With this notation, there is an edge between two vertices $i$ and $j$ if and only if $X_i$ lies in $\scap^p_{X_j}$ and vice versa.

Let $f_d(x)$ be the density function of a 1 dimensional marginal of $\umeasure{d-1}$, given by $$f_d(x) = \frac{\Gamma(d/2)}{\Gamma((d-1)/2)\sqrt{\pi}} (1 - x^2)^{(d-3)/2},$$
and define $\utaild{d}{u}:=\int_{u}^{1}f_d(x)dx$ to be the corresponding tail distribution function. Then $h(a)=\utaildinverse{d}{a}$ and $\uthreshold{p,d}=\utaildinverse{d}{p}$. If the dimension is clear from context, we drop $d$ from the notation and write $\utail{p}$ and $\utailinverse{p}$.

\paragraph{\textbf{Symmetric rearrangement of a subset.}} We associate with $\sphere$ a canonical center $\eb = (1,0,\ldots,0)$. Given a measurable subset $R \subset \sphere$, the \emph{symmetric rearrangement} of $R$, denoted by $\symm{R}$, is a \emph{spherical cap} centered at $\eb$ which has the same measure as $R$ does, i.e., $\symm{R} = \scap_{\eb}^{\sumeasure(R)}$. 

\paragraph{\textbf{Symmetric rearrangement of a function.}} Let $f : \sphere \rightarrow \bbR_+$ be a non-negative real valued measurable function. The \emph{symmetric rearrangement}\footnote{Also called \emph{symmetric decreasing rearrangement} in the references in this section.} of $f$ is defined to be a function $\symmfn{f}: \sphere \rightarrow \bbR_+$ satisfying 
$$\sumeasure\left(\left\{\bx \in \sphere: \symmfn{f}(x) \ge t \right\}\right)=\sumeasure\left(\left\{\bx \in \sphere: f(x) \ge t \right\}\right)$$
for each $t > 0$. Further, it is constant on each slice $\left\{ \bx \in \sphere : \inner{\bx}{\eb} = h\right\}$, $-1 \le h \le 1$, and is non-decreasing in $h$. For a rigorous definition of $\symm{f}$, we refer to
{\cite[Section~1.6, Section~7.1]{baernstein2019symmetrization}}. It is easy to see that for any measurable subset $R \subseteq \sphere$, the symmetric rearrangement of its indicator function is equivalent to the indicator of $\symm{R}$. 

In this section, we prove two propositions (Propositions~\ref{P.sym-dom.proba} and~\ref{P.sym-dom.MGF}) that are immediate corollaries of Theorem 2.2 in \cite{draghici2005rearrangement}, restated below. Intuitively, it says that certain functionals of functions are maximized by their symmetric rearrangements. To state it more formally, let $AL_2(\bbR_{+}^2)$ be the set of functions $f=f(x,y)$ such that $f(x,y)\le f(z,w)$ for all $x,y,z,w$ satisfying $x+y=z+w$ and $\max(x,y)\le\max(z,w)$. A continuous function $\psi: \bbR_{+}^d\to\bbR_{+}$ is said to be in $AL_2(\bbR_{+}^n)$ if after fixing any $d-2$ coordinates, $\psi$ is in $AL_2(\bbR_{+}^2)$ in the remaining two coordinates.  Note that for any $\psi\in C^2(\bbR_{+}^n)$, we have that $\psi\in AL_2(\bbR_{+}^n)$ if and only if $\partial_i\partial_j\psi\ge 0$ for every $i\neq j$.

\begin{theorem}[\cite{draghici2005rearrangement}, Theorem 2.2]\label{T.Rearrangement}
    Let $M=\sphere$, $f_i:M\to\bbR_+$ be $n$ non-negative functions, $\psi$ be a continuous function in $AL_2(\bbR_{+}^n)$, and $K_{i,j}:\bbR_+\to\bbR_+,1 \le i<j \le n$ be increasing functions. Let
    $$I[f_1,\ldots,f_n]:=\int_{M^n}\psi(f_1(\unif_1),\ldots,f_n(\unif_n))\prod_{i<j}K_{i,j}(\inner{\unif_i}{\unif_j}) d\mu(\unif_1)\ldots d\mu(\unif_n),$$
    where $d\mu$ is the uniform measure on $M$. Then, we have that
    $$I[f_1,\ldots, f_n]\le I[\symmfn{f_1},\ldots, \symmfn{f_n}].$$
\end{theorem}

\begin{remark}
    In the original statement in \cite{draghici2005rearrangement}, $K_{i,j}(\inner{\unif_i}{\unif_j})$ is replaced by $K_{ij}(d(\unif_i,\unif_j))$ where $K_{ij}$ are decreasing functions and $d(\cdot,\cdot)$ is the geodesic distance on $\sphere$. They are actually equivalent since $\inner{\unif_i}{\unif_j}=\cos(d(\unif_i,\unif_j))$ is a decreasing function in $d(\unif_i,\unif_j)$.
\end{remark}

Using this theorem, we establish two results for spherical RGGs. Recall that for two edge-weighted graphs $H_1$ and $H_2$ on the same node set, we write $H_1 \leq H_2$ if and only if $w^{H_1}_{i \nospace j} \leq w^{H_2}_{i \nospace j}$ for all $i\neq j$. The following proposition states that for any fixed edge-weighted graph $H_0$ and an edge-weighted random geometric graph $H$, the probability of the event $\{H\ge H_0\}$ under constraints $\unif_i\in A_i$ is optimized when $A_i$'s are \emph{spherical caps} centered at the same point.

\begin{proposition}\label{P.sym-dom.proba}
    Let $A_i \subset \sphere$, $i \in [n]$ be measurable, and let $\symm{A_i}$ be their \emph{symmetric rearrangements}. Let $H \sim \rswgg{n}$. Then for any edge-weighted graph $H_0$ on the same node set $[n]$,
    \begin{equation}\label{E.sym-dom.proba}
        \bbP\left(\{ H \geq H_0 \} \cap\bigcap_{i=1}^n\{X_i\in A_i\}\right)\le\bbP\left(\{ H \geq H_0 \} \cap\bigcap_{i=1}^n\{X_i\in \symm{A_i}\}\right) . 
    \end{equation}
\end{proposition}

\begin{proof}
    Take $f_i(X)=\indicatorfn{A_i}(X), \ i\in[n]$, and set $\psi(y_1,\ldots,y_n)=y_1y_2\cdots y_n$. Further, set $K_{i,j}(x)=\indicatorevent{x\ge w^{H_0}_{ij}}$. Then the left hand side of \eqref{E.sym-dom.proba} is equal to $I[f_1,\ldots,f_n]$ and right hand side is equal to $I[\symmfn{f_1},\ldots,\symmfn{f_n}]$, and the conclusion immediately follows from Theorem~\ref{T.Rearrangement}.
\end{proof}

\begin{proposition} \label{P.sym-dom.MGF}
    Let $f$ be a non-negative function on $\sphere$, and let $\unif_1,\ldots,\unif_n$ be \iid \ uniformly distributed on $\sphere$. Then for any $0\le p\le 1$ and $\theta\ge 0$, we have 
    $$\bbE\exp\left(\theta\left(\sum_{i=1}^{n}f(\unif_i)+\sum_{i<j}\indicatorevent{\inner{\unif_i}{\unif_j}\ge\uthreshold{p}}\right)\right)\le\bbE\exp\left(\theta\left(\sum_{i=1}^{n}\symmfn{f}(\unif_i)+\sum_{i<j}\indicatorevent{\inner{\unif_i}{\unif_j}\ge\uthreshold{p}}\right)\right).$$
\end{proposition}
\begin{proof}
    Take $f_i=f$ for all $i \in [n]$, and $\psi(y_1,\ldots,y_n)=\exp(\theta\sum_{i}y_i)$. Also set $K_{i,j}(x)=\exp(\theta \indicatorevent{x\ge\uthreshold{p}})$. It is a simple check to verify that the conditions of Theorem \ref{T.Rearrangement} are satisfied, and the result follows immediately.
\end{proof}

\subsection{Proof outline}\label{S.Proof outline}

The estimation of the probability of the RGGs being a complete graph and of the edge deviation events follows a similar approach. Instead of individually proving both Theorems~\ref{T.SClique} and~\ref{T.GClique} (resp., Theorems~\ref{T.SDeviation} and~\ref{T.GDeviation}), we use the coupling lemmas established in Section~\ref{S.Coupling} to reduce the analysis to only one of the two models. For some bounds, the spherical model turns out to be simpler to handle, meanwhile for others, we work with the Gaussian model.

The lower bounds for the probability of observing a complete graph are proved in Section~\ref{lower bound}. We first obtain an Erd\H{o}s-R\'enyi type lower bound, $p^{\binom{n}{2}}$, for the spherical model using a purely geometric argument. Note that this lower bound does not depend on the dimension $d$. The remaining bounds in Equations~\eqref{E.UCliqueLB} and~$\eqref{E.GCliqueLB}$ are proved by considering the Gaussian model, which is easier to handle with our method. To obtain these bounds, we consider the event that the first coordinate of all the vectors $\gauss_1, \ldots,\gauss_n$ have a small \emph{bias} of order $\sqrt[4]{d\log n}$. Conditioned on this event,  we show that the GRGG is a clique with a positive probability. We apply a similar idea to establish a lower bound on edge deviation probability, see Section~\ref{S.edge lower bound}.

The upper bounds are proved in Section~\ref{S.clique upper bound} and are more involved. We observe that conditioned on the GRGG being a complete graph, the statistic $S=\sum_{i\neq j}\inner{\gauss_i}{\gauss_j}$ is atypically large. This statistic is easier to handle since it is the sum of products of independent standard Gaussian random variables, and we use a Bayesian argument to get one upper bound. We obtain another upper bound by considering the SRGG on $n$ vertices to be obtained by a graph process, where the vertices are added one at a time. This gives us a corresponding subset valued random process $(A_k)_{k \ge 0}$ starting from $A_0 = \sphere$, and conditioned on $\unif_{1},\ldots,\unif_{k}$, $A_k$ is the subset of $\sphere$ where $\unif_{k+1}$ can be in so that the SRGG on the first $k+1$ vertices is still a complete graph. By construction, $A_{k+1} = A_k \cap \scap^p_{\unif_{k+1}}$. We use symmetric rearrangements of subsets to transform the process $(A_k)_{k \ge 0}$ to a spherical cap valued random process. This reduced process is more tractable, and we prove one of the bounds in Theorems~\ref{T.SClique} and~\ref{T.GClique} by analyzing how the caps shrink with the arrival of each new vertex.

For the edge deviation event $\uedgedeviation{\deviationrate}$, we first obtain a bound on the moment generating function (MGF) of the number of edges in the RGG. Defining $$f_k(\cdot) = \sum_{i=1}^{k} \indicatorfn{\scap^p_{\unif_i}}(\cdot),$$ the number of edges in the SRGG can be written as $|E(G)| = \sum_{i=1}^n f_{i-1}(\unif_i)$. To get an upper bound on the MGF of $|E(G)|$, we use symmetric rearrangement of functions to transform $f_k$ into a more tractable process, and then bound the MGF of $\setsize{E(G)}$ by the MGF of this reduced process. Using exponential Markov's inequality then completes the proofs.

\section{Estimating Clique Probability}\label{S.Estimating Clique Probability}

This section contains the proofs of Theorems~\ref{T.SClique} and~\ref{T.GClique}, along with the statements and proofs of various lemmas and propositions that we use. 
\subsection{Lower bounds}\label{lower bound}

In this section, we obtain the lower bounds in Theorems~\ref{T.SClique} and ~\ref{T.GClique}. The three different bounds in the exponent are derived via different techniques. First, we prove a lower bound that does not depend on the dimension $d$. As a consequence, we see that the clique probability for the spherical RGG is at least as large as that for the Erd\H{o}s-R\'enyi graph, which matches the intuition that the RGGs are more clustered than the Erd\H{o}s-R\'enyi graph.

\begin{proposition}
    \label{P.ERlowerbound}
    For all $n,d \ge 2$ and $0\le p\le\frac{1}{2}$,$$\bbP(\uclique)\ge p^{\binom{n}{2}}\ge\exp(-cn^2).$$
\end{proposition}

We prove the above proposition by using the following result about the spherical model. We drop the subscript and superscript from $\umeasure{d-1}$ for simplicity.
\begin{lemma} \label{L.+correlation}
Let $\sumeasure$ denote the uniform probability measure on $\sphere$. Suppose $A\subset \sphere$ is convex, and $\bx \in A$. Then, 
\begin{align}
\mu(A\cap\scap_{\bx}^p) \ge p\cdot\sumeasure(A). 
\end{align}
\end{lemma}
\begin{proof}
    Recall that $\utail{x}=\bbP(\unif_1^{(1)}>x)$ is the tail distribution function of 1 dimensional marginal of $\sumeasure$.  We construct a contraction mapping $\varphi=\varphi_{p,x} :\sphere \setminus \{-x\}\to\scap_{\bx}^p$ as follows. For every $y\in\sphere \setminus \{-x\}$, set $\varphi_{p,x}(y)$ to be the unique point $z$ on the geodesic segment connecting $x$ and $y$ s.t. $\inner{x}{z}=\utailinverse{p\cdot \utail{\inner{x}{y}}}$. Then $\sumeasure \circ \varphi=p\cdot\sumeasure$, and since $A$ is convex with $x \in A$, $\varphi(A)\subset A$. So we immediately conclude that $$\sumeasure(A\cap\scap_{\bx}^p)\ge\sumeasure(\varphi(A))=p\cdot\sumeasure(A).$$
\end{proof}
\begin{proof}[Proof of Proposition~\ref{P.ERlowerbound}]\label{Pr.proofERlowerbound}

Define the sets $A_k = \bigcap_{i=1}^k \scap_{X_i}^p$, and consider the events 
\[
\cE_k = \{\inner{X_i}{X_j}\ge\uthreshold{p} \  \forall \ 1 \leq i \neq j \leq k\}, \ 1 \leq k \leq n.
\]
Note that $\cE_n = \uclique$, and 
\[
\bbP(\uclique) = \prod_{i=1}^{n-1} \bbP(\cE_{i+1}|\cE_{i}).
\]
Further, observe that $\mathbb P(\mathcal E_{k+1}|\mathcal E_k)=\mathbb E[\unif_{k+1} \in A_k |\mathcal E_k] =\mathbb E[\mu(A_{k})|\mathcal E_k]$. Using Lemma~\ref{L.+correlation} recursively, we obtain that $\mathbb E[\mu(A_{k})|\mathcal E_k] \ge p^k$, which we plug into the previous equation to conclude that
\[
\bbP(\uclique) \ge p^{\binom{n}{2}} \ge p^{n^2}.
\]
\end{proof}

Note that the result above does not depend on $d$. But for $k$ large enough, a typical $A_k$ is too ``small'', in the sense that $A_{k+1}$ has nearly the same measure as $A_k$ instead of $p \sumeasure(A_k)$. We obtain a dimension dependent bound in this regime by introducing a biasing argument, where we condition on the event that all the points have a slightly large first coordinate. The technique can be used in both spherical case and Gaussian case. We choose the Gaussian case for simplicity. 

\begin{proposition}\label{P.biaslowerbound}
    There exists constants $\varepsilon,C>0$ such that for all $n,d\ge 2$ and $\log n<\varepsilon d$,
    $$\bbP(\gclique)\ge\exp(-C n \sqrt{d\log n}).$$
\end{proposition}

\begin{proof}
    Suppose $\gauss_i=(\gauss_i^{(1)},\ldots,\gauss_i^{(d)})$, where $\gauss_i^{(k)}$'s are \iid \ standard Gaussian random variables. For any constant $C > 0$, define the biased event $\mathcal B_C:=\{\gauss_i^{(1)}\ge (\gthreshold{p}+C\sqrt{d\log n})^{1/2} \  \forall 1 \le i \le n\}$. Conditionally on $\mathcal B_C$ happening, $\mathcal E$ fails only if 
    $$\sum_{k=2}^{d}\tilde X_i^{(k)}\tilde X_j^{(k)}<\gthreshold{p}-\tilde X_i^{(1)}\tilde X_j^{(1)}\le -C\sqrt{d\log n}$$
    for some $i,j\in[n]$.
    Since $(\gauss_i^{(j)}\gauss_k^{(j)})_{2\le j\le d}$ are independent, mean zero, sub-exponential random variables, we can apply Bernstein's inequality (see Theorem 2.8.1 in~\cite{vershynin2018high}) to get:
    $$\mathbb P\left(\left|\sum_{k=2}^{d}\tilde X_i^{(k)}\tilde X_j^{(k)}\right|\ge C\sqrt{d\log n} \right)\le 2\exp\left[-c\min\left(C^2\log n,C\sqrt{d\log n}\right)\right],$$
    which is bounded by $2n^{-3}$ if $\log n<\varepsilon d$ for sufficiently small $\varepsilon$ and appropriately chosen constant $C$. By taking an union bound on all possible pairs $(i,j)$, we have that $$\bbP\left(\mathcal E|\mathcal B_C\right)\ge 1-\binom{n}{2}2n^{-3}\ge \frac{1}{2}.$$
    This gives $$\mathbb P(\mathcal E)\ge \frac{1}{2}\mathbb P(\mathcal B_C)\ge \frac{1}{2}\gtail{(\gthreshold{p}+C\sqrt{d\log n})^{1/2}}^n\ge \exp(-Cn\sqrt{d\log n}).$$ 
\end{proof}

If $n$ is exponentially large in $d$, then that union bound used in the proof above will be redundant, and we can establish a better bound by adding a small bias to all coordinates.

\begin{proposition} \label{P.stronger bias}
    There exists a constant $C>0$ such that for all $n,d\ge 2$,
    $$\bbP(\gclique)\ge\exp(-C nd).$$
\end{proposition}
\begin{proof}
    Consider a different event $\mathcal E_0':=\{\gauss_i^{(j)}\ge (\gthreshold{p}/d)^{1/2} \ \forall 1\le i\le n, 1\le j\le d\}$ which immediately implies $\gclique$. 
    Recall that $\gthreshold{p}/d=O(d^{-1/2})$ so we have $$\bbP(\gclique)\ge\gtail{((\gthreshold{p}/d)^{1/2}}^{nd}\ge 2^{-(1+o(1))nd}\ge \exp(-Cnd).$$
\end{proof}

Combining the results of Propositions~\ref{P.ERlowerbound}, \ref{P.biaslowerbound} and \ref{P.stronger bias} via the coupling result (Lemma~\ref{L.coupling}) gives us the desired lower bounds in Theorems~\ref{T.SClique} and~\ref{T.GClique}.

\subsection{Upper bounds}\label{S.clique upper bound}

We can infer from the proof of Theorem \ref{P.biaslowerbound} that in some regimes (e.g. $d\gtrsim \sqrt n$), the event $\gclique$ happens if there is a small bias for a majority of the Gaussian vectors. To capture a similar phenomenon from a macroscopic perspective, we introduce a statistic $S:=\sum_{1\le i\neq j\le n}\inner{\gauss_i}{\gauss_j}$. The following lemmas cover some properties of this statistic.

\begin{lemma}\label{L.S is biased}
    There exists a constant $c_1=c_p>0$ such that for all $n,d\ge 2$,
    $$\mathbb E[S|\gclique]\ge c_1n^2\sqrt d.$$
\end{lemma}
\begin{proof}
    If $p<1/2$, we can deduce from Lemma \ref{L.estimate threshold} that 
    $$\bbE[S|\gclique]\ge (n^2-n)\gthreshold{p}\ge c_1n^2\sqrt d.$$
    For the case $p=1/2$, $\gthreshold{p} = 0$, and we use a geometric argument to obtain the result. Note that in this case, the two RGG models coincide. So we first consider the spherical model and prove that $\bbP(\uclique|\inner{\unif_1}{\unif_2}=x)$ is monotonically increasing in $x$. For $0\le x<x' \le 1$, fix $\unif_1,\unif_2$ so that $\inner{\unif_1}{\unif_2}=x$, and choose $\unif_2'$ to lie on the geodesic connecting $\unif_1$ and $\unif_2$ s.t. $\inner{\unif_1}{\unif_2'}=x'$.  By construction, $\unif_2'$ is a convex combination of $\unif_1$ and $\unif_2$, so if $\unif_3,\ldots,\unif_n$ are sampled uniformly at random on $\sphere$, independently of $\unif_1,\unif_2$ and $\unif_2'$, the event that $(\unif_1,\unif_2,\unif_3,\ldots,\unif_n)$ forms a clique implies that $(\unif_1,\unif_2',\unif_3,\ldots,\unif_n)$ also forms a clique. Then we get
    $$\bbP(\uclique|\inner{\unif_1}{\unif_2}=x')\ge \bbP(\uclique|\inner{\unif_1}{\unif_2}=x),$$
    and hence $$\bbE[\inner{\unif_1}{\unif_2}|\uclique]\ge\bbE[\inner{\unif_1}{\unif_2}|\inner{\unif_1}{\unif_2}\ge 0].$$
    Finally, observe that $$\bbE\left[\inner{\gauss_1}{\gauss_2}\Bigg|\gclique\right]=\bbE\norm{\gauss_1}\cdot\bbE\norm{\gauss_2}\cdot\bbE[\inner{\unif_1}{\unif_2}|\uclique]\ge\bbE\left[\inner{\gauss_1}{\gauss_2}\Bigg|\inner{\gauss_1}{\gauss_2}\ge 0\right].$$
    Since $\inner{\gauss_1}{\gauss_2}$ has asymptotic distribution $\mathcal N(0,d)$, the term on the right is bounded below by $c_1\sqrt d$ for some $c_1>0$. Replacing $(\gauss_1,\gauss_2)$ by any pair $(\gauss_i,\gauss_j)$, $i \neq j$ yields the same bound. Summing up over all $i,j$ gives $\bbE[S|\gclique]\ge c_1n^2\sqrt d$.
\end{proof}
Since $S$ is the sum of \iid \ sub-exponential random variables, it concentrates around its mean, and we have the following estimate using Bernstein’s inequality.

\begin{lemma} \label{L.concentration of S}
    There exists a constant $c_2>0$ such that for all $n,d\ge 2$ and $t>0$,
    $$\bbP(|S|\ge t)\le 2\exp\left(-c_2\min\left(\frac{t^2}{n^2d},\frac{t}{n}\right)\right).$$
\end{lemma}
\begin{proof}
    Note that we can decompose the sum as
    $$S=\sum_{k=1}^{d}\left(\sum_{1\le i\neq j\le n}\gauss_i^{(k)}\gauss_j^{(k)}\right)=:\sum_{k=1}^{d}\Gaussian^{(k)},$$ 
    where $\Gaussian^{(k)}$ are \iid \ mean 0 random variables with the same distribution as $\Gaussian:=(\sum_{i=1}^n \Gaussian_i)^2-\sum_{i=1}^{n}\Gaussian_i^2$, in which $\Gaussian_i$ are \iid \ standard normal variables. For every positive integer $q$, the following moment estimate on $\Gaussian$ holds: 
    $$\norm{\Gaussian}_{L_q}\le\left\|\left(\sum_{i=1}^n \Gaussian_i\right)^2\right\|_{L_q}+\sum_{i=1}^{n}\left\|\Gaussian_i^2\right\|_{L_q}\le 6nq.$$
    which coincides with property (b) in Proposition 2.7.1 in~\cite{vershynin2018high}. It follows that $\Gaussian$ is a sub-exponential variable with $\norm{\Gaussian}_{\psi_1}\lesssim n$. Then Lemma~\ref{L.concentration of S} follows from Bernstein's inequality.
\end{proof}

Lemma \ref{L.S is biased} and Lemma \ref{L.concentration of S} together give the following upper bound.

\begin{proposition}\label{bayes upper bound}
    There exists a constant $c>0$ such that for all $n,d\ge 2$,
    $$\bbP(\gclique)\le\exp\left(-c\min\left(n^2,n\sqrt{d}\right)\right)$$
\end{proposition}
\begin{proof}%[Proof of Proposition~\ref{bayes upper bound}]
    We use the following extension of Markov's inequality for the tail of positive random variables. Suppose $A$ is measurable and $\xi$ is a real valued random variable such that $a =\mathbb{E} [\xi | A]$ is finite. Then for any $b < a$,
    \begin{align}
    (a - b) \mathbb{P} (A) =\mathbb{E} [(\xi - b) 1_A] \leq \bbE (\xi-b)^+  \Rightarrow \mathbb{P} (A) \leq \frac{1}{a - b} \mathbb{E} (\xi-b)^+, \label{E.markov-ineq.event_cond-expect}
    \end{align}
    where $(x)^+ = \max(x,0)$.
    Take $A = \tilde{\mathcal{E}}$, $\xi = S$ and $b = \frac{a}{2}$. By \Cref{L.S is biased}, since $a = \bbE[S|\gclique]$, $b=\ge c_1n^2\sqrt d/2$. Using \Cref{L.concentration of S} in the first line below, the RHS of \eqref{E.markov-ineq.event_cond-expect} can be bounded by
    \begin{align*}
        \bbE(S-b)_+&\le\int_{b}^{+\infty}(2\exp(-c_2t/n)+2\exp(-c_2t^2/n^2d)) dt\\
        &\le\frac{2n}{c_2}\exp(-c_2b/n)+\frac{n^2d}{cb}\exp(-c_2b^2/n^2d)\\
        &\le C(n+\sqrt d)\exp\left(-c\min\left(n^2,n\sqrt{d}\right)\right)
    \end{align*}
    for some $c > 0$. Therefore 
    $$\bbP(\gclique)\le \frac{1}{a-b}\bbE[(S-b)_{+}] \le \exp \left(-(c+o(1))\min\left(n^2,n\sqrt{d}\right) \right).$$
\end{proof}

\begin{remark}
    In principle, one can consider the same argument in the spherical model and use the statistic $\sum_{i\neq j}\inner{\unif_i}{\unif_j}$ to derive a similar upper bound. But in the spherical model, $\sum_{i\neq j}\inner{\unif_i}{\unif_j}$ is not the sum of \iid \ random variables, which presents a minor obstacle. So we chose to obtain the bounds for the Gaussian model and rely on Lemma~\ref{L.coupling} to get bounds of the same order for the spherical model.
\end{remark}
Another way of dealing with the upper bound is to analyze the convex set-valued random process $A_k$, defined in the proof of Theorem \ref{P.ERlowerbound}. If we define $H_k$ be the induced subgraph of $\srgg(n,p)$ on the last $k$ vertices $n-k+1,\ldots,n$. Then we have the inequality
\begin{align*}
    \bbP(\uclique)\le\bbP\left(\{H_{n-k}\ge \uthreshold{p,d}\1_{n-k}\}\cap\bigcap_{i=k+1}^n\{\unif_i\in A_k\}\right)
\end{align*}
This reduces analyzing $\uclique$ to the study of the random process $A_k$, which is still quite complicated. The following proposition allows us to compare $A_k$ with its symmetric rearrangement $\symm{A_k}$, which is easier to analyze.
Define $\scapprocess_0=\sphere$, and iteratively define $\scapprocess_{k+1}:=\symm{(\scapprocess_k\cap\scap_{X_{k+1}}^p)}$. 
It is a sequence of spherical caps all centered at $\eb$. 
\begin{proposition}\label{P.upper bound by cap process}
    For any $0\le k\le n$, we have
    \begin{equation}
    \label{E.upper bound by cap process}
    \bbP(\uclique)\le\bbP\left(\{H_{n-k}\ge\uthreshold{p,d}\1_{n-k}\}\cap\bigcap_{i=k+1}^{n}\{\unif_i\in\scapprocess_k\}\right)
    \end{equation}
\end{proposition}
\begin{proof}
    Note that for any $0\le k\le n-1$,
    \begin{align*}
        &\bbP\left(\{H_{n-k}\ge \uthreshold{p,d}\1_{n-k}\}\cap\bigcap_{i=k+1}^{n}\{\unif_i\in\scapprocess_{k}\}\right)\\
        ={}&\bbP\left(\{H_{n-k-1}\ge \uthreshold{p,d}\1_{n-k-1}\}\cap\bigcap_{i=k+1}^{n}\left\{\unif_i\in\left(\scapprocess_{k}\cap\scap^p_{\unif_{k+1}}\right)\right\}\right)\\
        \le{}&\bbP\left(\{H_{n-k-1}\ge \uthreshold{p,d}\1_{n-k-1}\}\cap\bigcap_{i=k+2}^{n}\left\{\unif_i\in\left(\scapprocess_{k}\cap\scap^p_{\unif_{k+1}}\right)\right\}\right)\\
        \le{}&\bbP\left(\{H_{n-k-1}\ge \uthreshold{p,d}\1_{n-k-1}\}\cap\bigcap_{i=k+2}^{n}\{\unif_i\in\scapprocess_{k+1}\}\right).
    \end{align*}
    Here, in the first inequality we drop $\left\{\unif_{k+1}\in\left(\scapprocess_{k}\cap\scap^p_{\unif_{k+1}}\right)\right\}$, and in the second inequality, we apply Proposition \ref{P.sym-dom.proba}. This proves that the right side of \eqref{E.upper bound by cap process} is non-decreasing in $k$. Note that the two side are equal when $k=0$, hence \eqref{E.upper bound by cap process} holds for all $0\le k\le n$.
\end{proof}

Defining $\eta_k:=\sumeasure(\scapprocess_k)$, and dropping the event $\{H_{n-k}\ge\uthreshold{p,d}\1_{n-k}\}$ in the RHS of~\eqref{E.upper bound by cap process}, we get that \begin{equation}
    \bbP(\uclique)\le\bbE(\eta_k^{n-k}).\label{E.upper bound by cap process 2}
    \end{equation}
$(\eta_k)_{k \ge 0}$ is a random, decreasing process, and the following lemma quantifies the difference $\eta_k - \eta_{k+1}$. Recall that $h(\eta)$ is the \emph{altitude} of a spherical cap with measure $\eta.$    

\begin{lemma}\label{L.fast decay} 
    For any $\eta \ge 0$, let $\eta' = \umeasure{d-1}(\scap^\eta_{\eb} \cap \scap^p_Y)$, where $Y$ is uniformly distributed on $\sphere$. Then there exists $c,c_0>0$ s.t. for any $\delta\in(0,c)$ and $\eta\in[\exp(-c_0d\delta^2/2),1/2]$,
    $$\bbP\left(\eta-\eta'\ge c\delta\exp\left(-\frac{d}{2}(1+3\delta^2)h(\eta)^2\right)\right)\ge 1-\exp(-d\delta^2/2).$$
\end{lemma}

To prove Lemma \ref{L.fast decay} we use the following distributional approximation result of Sodin \cite{sodin2007tail}.

\begin{lemma}\label{sodin's Lemma}
    There exists constants $C,C_1,C_2>0$  and a sequence $\varepsilon_d\searrow 0$ such that for all $t \in [0,C]$,
    $$(1-\varepsilon_d)\gtail{\sqrt d t}\exp(-C_1dt^4)\le\utail{t}\le(1+\varepsilon_d)\gtail{\sqrt dt}\exp(-C_2dt^4).$$
    In particular, if $t=O(d^{-1/4})$, then 
    $$\utail{t}\asymp\frac{1}{\sqrt dt}\exp(-dt^2/2).$$
\end{lemma}

\begin{proof}[Proof of Lemma \ref{L.fast decay}]
Note that $\eta'$ depends only on $Y^{(1)}$, the first coordinate of $Y$. Recall the density of $Y^{(1)}$ is $$f_d(x) = \frac{\Gamma(d/2)}{\Gamma((d-1)/2)\sqrt{\pi}} (1 - x^2)^{(d-3)/2}.$$ 
From Lemma \ref{sodin's Lemma}, we know that $\bbP(Y^{(1)}>\delta)\le\exp(-d\delta^2/2)$. Next we show that if $Y^{(1)} \le\delta$, then $\eta-\eta'\ge c\delta\exp\left(-\frac{d}{2}(1+3\delta^2)h(\eta)^2\right)$. Since $\eta'$ is monotonically increasing in $Y^{(1)}$, it suffices to consider $Y=(\delta, -\sqrt{1-\delta^2},0,\ldots,0)$. Note that $\{(x,y,\ldots)\in\mathbb S^{d-1}:x>h(\eta),~y >\delta x /\sqrt{1-\delta^2}\} \subseteq \scap^\eta_{\eb} \setminus \scap^p_{Y}$. Further, the 2 dimensional marginal of $\umeasure{d-1}$ on the first two coordinates has density function $$f_d(x,y)=\frac{d-2}{2\pi}(1-x^2-y^2)^{(d-4)/2}.$$
By direct computation, when $Y^{(1)}\le\delta$ we have
\begin{align*}
    \sumeasure(\scap^\eta_{\eb}\setminus\scap^p_{Y^{(1)}})&\ge\int_{x^2+y^2\le1}\mathrm 1_{\{x>h(\eta),~y>\delta/\sqrt{1-\delta^2}x\}}f_d(x,y)dxdy\\
    &\ge cd\int_{h(\eta)}^{1/2}\left(\int_{\delta x/\sqrt(1-\delta^2)}^{\sqrt{1-x^2}}(1-x^2-y^2)^{(d-4)/2}dy\right)dx\\  &=cd\int_{h(\eta)}^{1/2}(1-x^2)^{(d-3)/2}\left(\int_{\delta x/\sqrt{(1-\delta^2)(1-x^2)}}^1(1-t^2)^{(d-4)/2}dt\right)dx\\
    &\ge cd\int_{h(\eta)}^{1/2}(1-x^2)^{d/2}\delta x(1-2\delta^2x^2)^{d/2}dx\\
    &\ge c\delta\left(1-(1+2\delta^2)h(\eta)^2\right)^{d/2}\\
    &\ge c\delta\exp\left(-\frac{d}{2}(1+3\delta^2)h(\eta)^2\right).
\end{align*}
For the last two inequalities to hold, we need $h(\eta)<c'\delta$ for some constant $c'$ small enough. This can be achieved by taking $c_0$ to be sufficiently small in the assumption, which completes the proof.
\end{proof}
Using Lemma~\ref{L.fast decay}, we obtain the following two rates of decay of the sequence $(\eta_n)_{n \ge 1}$ as $n$ goes to infinity, for different regimes of $d$. 
\begin{proposition}\label{P.shrinking 1}
    For a sufficiently small constant $c>0$ and $\log n \lesssim d \lesssim n^{4/3}$,
    $$\bbP\left(\eta_n\ge\exp\left(-c\sqrt{d}\cdot\frac{\log n-\log d^{3/4}}{\sqrt{\log n}}\right)\right) \le \exp \left(-cn\sqrt{d\log n}\right).$$
\end{proposition}

\begin{proof}
    We first select a constant $0<c<<1$ and $\delta=c\sqrt[4]{\log n/d}$. Let $I\subset[n]$ contains all index $i$ such that the first coordinate of $\unif_i$ is smaller than $\delta$. Since $I$ is a Bernoulli process with success parameter $p_0 \ge 1-\exp(-d\delta^2/2)\ge 1-\exp(-c\sqrt{d\log n})$,
    $$\bbP(|I|\le n/2)\le\sum_{i=0}^{n/2}p_0^i(1-p_0)^{n-i}\binom{n}{i}\le 2^n(1-p_0)^{n/2}\le\exp(-cn\sqrt{d\log n}).$$
    Let $N=(\log n-\log d^{3/4})/\sqrt{\log n}$. On the event $|I|\ge n/2$ we prove $\eta_n<\exp(-c\sqrt dN)$. Since $\eta_1=p\le 1/2$,  we may assume $\eta_i\in[\exp(-c\sqrt{d\log n}),1/2]$ for any $1\le i\le n$. From the proof of Lemma \ref{L.fast decay}, we deduce that for $i\in I$, $\eta_i-\eta_{i+1}>c\delta\exp\left(-\frac{d}{2}(1+3\delta^2)h(\eta_i)^2\right).$ Denote $I_k=I\cap\{i:\eta_i\in(e^{-k-1},e^{-k}]\}$ for any integer $1\le k\le \sqrt{d\log n}$. Using Lemma \ref{sodin's Lemma}, we get that for any $i\in I_k$, 
    \begin{equation}e^{-k-1}\le C\exp\left(-\frac{d}{2}h(\eta_i)^2\right).\label{shrinking 1}\end{equation}
    Then $$\eta_{i}-\eta_{i+1}\ge c\delta e^{-(1+3\delta^2)k}$$ and therefore
    \begin{equation}|I_k|\le\left\lceil\frac{e^{-k}}{c\delta e^{-(1+3\delta^2)k}}\right\rceil\le 1+\frac{1}{c\delta}e^{3\delta^2k}.\label{shrinking 2}\end{equation}
    Summing the above inequality for $0\le k\le c\sqrt{d}N$ gives
    \begin{equation}\sum_{k=0}^{c\sqrt dN}|I_k|\le c\sqrt dN+1+\frac{e^{3c\delta^2\sqrt dN}}{c\delta(e^{3\delta^2}-1)}\le \frac{n}{2}.\label{shrinking 3}\end{equation}
    The last inequality holds for sufficiently small $c>0$. Therefore, $\eta_n<\exp(-c\sqrt d N)$ when $\setsize{I}\ge n/2$.
\end{proof}
\begin{proposition}\label{P.shrinking 2}
    For sufficiently small constant $c>0$, and $d \lesssim \log n$,
    $$\bbP\left(\eta_n\ge\exp\left(-cd\right) \right) \le \exp \left(-cnd\right).$$
\end{proposition}
\begin{proof}
    The proof is similar to that or Proposition~\ref{P.shrinking 1}. We select a small constant $\delta<<1$, and define $I$ and $I_k$ for $1\le k\le d$ as before. In this regime, $I$ is again a Bernoulli process, but with parameter $p_0\ge1-\exp(-cd)$. Therefore $\bbP(\setsize{I}\le n/2)\le\exp(-cnd)$. \eqref{shrinking 1} and \eqref{shrinking 2} still holds, and we can sum over $k$ to get that
    $$\sum_{k=0}^{cd}\setsize{I_k}\le 1+cd+\frac{e^{3c\delta^2d}}{c\delta(e^{3\delta^2}-1)}\le \frac{n}{2}$$ 
    holds for a small enough constant $c$. Thus $\eta_n\le e^{-cd}$ if $\setsize{I}\ge n/2$.
\end{proof}
Finally, we complete the proof of the upper bounds in Theorems \ref{T.SClique} and \ref{T.GClique}. Substituting Proposition \ref{P.shrinking 1} into \eqref{E.upper bound by cap process}, when $\log n \lesssim d \lesssim n^{4/3}$, we obtain
    \begin{align*}
        \bbP(\uclique)&\le\bbE (\eta_{n/2}^{n/2})\le \bbP\left(\eta_n\ge\exp\left(-c\sqrt{d}\cdot\frac{\log n-\log d^{3/4}}{\sqrt{\log n}}\right)\right)+\exp\left(-c\sqrt{d}\cdot\frac{\log n-\log d^{3/4}}{\sqrt{\log n}}\cdot\frac{n}{2}\right)\\
        &\le \exp\left(-cn\sqrt{d\log\frac{n^{1-o(1)}}{d^{3/4}}}\right).
    \end{align*}
    Similarly, in the regime $d \lesssim \log n$, we use Proposition \ref{P.shrinking 2} to get
    \begin{align*}
        \bbP(\uclique)\le\bbE \left(\eta_{n/2}^{n/2}\right)\le \bbP\left(\eta_n\ge\exp\left(-cd\right)\right)+\exp\left(-cd\cdot\frac{n}{2}\right)\le \exp\left(-cnd\right).
    \end{align*}
    The results above, along with Proposition~\ref{bayes upper bound} and Lemma~\ref{L.coupling}, complete the proof of the upper bounds in Theorems~\ref{T.SClique} and~\ref{T.GClique}.

\begin{remark}Since we lose some information during the symmetric rearrangement transformations, the upper bound we finally get is not tight when $n^{4/3} \lesssim d \lesssim n^2$. The Bayesian argument gives a matching upper bound when $n^2 \lesssim d$. Meanwhile, in the intermediate regime, it is not immediately clear what the correct order of decay should be, and is left as an open problem for the future.
\end{remark}

\section{Estimating Edge Deviation Probability}\label{S.Estimating Edge Deviation Probability}
In this section we present the proofs of Theorems \ref{T.SDeviation} and~\ref{T.GDeviation}, our main results on the large deviation for edges of the spherical and Gaussian RGG models.

\subsection{Lower bounds} \label{S.edge lower bound}

The first lower bound is a corollary of Proposition~\ref{P.ERlowerbound}, the Erd\H{o}s-R\'enyi type lower bound for the clique probability. It follows from noting that conditioned on the existence of a large enough clique, the event $\uedgedeviation{\deviationrate}$ happens with a probability bounded away from zero.
\begin{proposition}\label{P.ERlowerbound2}
    For all $n$ and $d\ge 2$,    $$\bbP(\uedgedeviation{\deviationrate})\ge \exp\left[-\left(\frac{p\log\frac{1}{p}}{2(1-p)}\deviationrate+o(1)\right)n^2\right].$$    
\end{proposition}
\begin{proof}
    Let $\alpha=\sqrt{p\deviationrate/(1-p)}$, and let $\uclique_{\ceil{\alpha n+1}}$ be as defined in the proof of Proposition~\ref{Pr.proofERlowerbound}. We can infer from the proof that 
    \begin{equation}
    \label{not interesting}
    \bbP(\uclique_{\ceil{\alpha n+1}})\ge p^{\binom{\ceil{\alpha n+1}}{2}}\ge\exp\left(-\frac{p\log\frac{1}{p}}{2(1-p)}\deviationrate n^2-n\right).
    \end{equation}
    Note that conditioned on $\uclique_{\ceil{\alpha n+1}}$ happening, the connection probability of edges not included in this small clique remains equal to $p$. That implies if $G \sim \srgg(n,p)$,
    $$\bbE\left[\setsize{E(G)}\Big|\uclique_{\ceil{\alpha n+1}}\right]\ge \binom{\ceil{\alpha n+1}}{2}+p\left[\binom{n}{2}-\binom{\ceil{\alpha n+1}}{2}\right]\ge p(1+\deviationrate)\binom{n}{2}+1.$$
    Hence $$\bbP\left(\setsize{E(G)}\ge p(1+\deviationrate)\binom{n}{2}\Big|\uclique_{\ceil{\alpha n+1}}\right)\ge\frac{1}{\binom{n}{2}}\left[\bbE\left[\setsize{E(G)}\Big|\uclique_{\ceil{\alpha n+1}}\right]-p(1+\deviationrate)\binom{n}{2}\right]\ge\frac{1}{n^2}.$$
    Combining with~\eqref{not interesting} completes the proof.
\end{proof}
For the second lower bound, we apply the biasing argument like we did in the proof of Proposition~\ref{P.biaslowerbound}. Again, we prove the result for the Gaussian model and use Lemma~\ref{deviation coupling} to conclude the result for the spherical model.
\begin{proposition}\label{P.biaslowerbound2}
    For all $n$ and $d\ge 2$, $$\bbP(\gedgedeviation{\deviationrate})\ge\exp\left[-\left(\frac{\gtailinverse{p}-\gtailinverse{p(1+\deviationrate)}}{2}+o(1)\right)n\sqrt d\right].$$
\end{proposition}
\begin{proof}
    Suppose $\gauss_i=(\gauss_i^{(1)},\ldots,\gauss_i^{(d)})$. Let $C:=C_{n,d,p,\deviationrate}=\sqrt{\gthreshold{p,d}-\gthreshold{p(1+\deviationrate)+1/n,d-1}}$ and define the biased event $\mathcal B_C:=\{\gauss_i^{(1)}>C,~\forall~1\le i\le n\}$.
    Then for all $1\le i<j\le n$ we have
    $$\bbP\left(\inner{\gauss_i}{\gauss_j}\ge \gthreshold{p,d}\Big|\mathcal B_C\right)\ge \bbP\left(\sum_{k=2}^{d}\gauss_i^{(k)}\gauss_j^{(k)}\ge\gthreshold{p,d}-C^2\right)=\left(p(1+\deviationrate)+1/n\right) \wedge 1.$$
    Let $G \sim \grgg(n,p)$. By summing over all $i,j$ we obtain
    $$\bbE\left[\setsize{E(G)}\Big|\mathcal B_C\right]\ge p(1+\deviationrate)\binom{n}{2}+1.$$
    Similar to the proof of Proposition~\ref{P.ERlowerbound2}, we get that 
    \begin{equation}
    \label{E.biaslowerbound2-2}
    \bbP\left(\setsize{E(G)}\ge p(1+\deviationrate)\binom{n}{2}\Big|\mathcal B_C\right)\ge \frac{1}{n^2}.
    \end{equation}
    Since $1/n \rightarrow 0$ as $n \rightarrow \infty$, we can deduce from Lemma~\ref{L.estimate threshold} that
    \begin{equation}
    \label{E.biaslowerbound2-3}
    C^2=\gthreshold{p,d}-\gthreshold{p(1+\deviationrate)+1/n,d-1}=\left[\gtailinverse{p}-\gtailinverse{p(1+\deviationrate)}+o(1)\right]\sqrt d.
    \end{equation}
    Using the standard estimation for Gaussian tails, we obtain
    \begin{equation}
    \label{E.biaslowerbound2-4}
    \bbP(\mathcal B_C)=\gtail{C}^n\ge\exp\left[-(1+o(1))C^2n/2\right].
    \end{equation}
    Combining~\eqref{E.biaslowerbound2-2},~\eqref{E.biaslowerbound2-3} and~\eqref{E.biaslowerbound2-4} concludes the proof.
\end{proof}

Combining Proposition~\ref{P.ERlowerbound2} and Proposition~\ref{P.biaslowerbound2} together with Lemma~\ref{deviation coupling} gives the lower bounds in Theorems~\ref{T.SDeviation} and~\ref{T.GDeviation}.

\subsection{Upper bounds}\label{S.Upper bound}

We consider the number of edges $\setsize{E(G)}$ as the final term of the following stochastic process. Denote by $\setsize{E_k(G)}$ the number of edges between $\unif_1,\ldots,\unif_k$, and define $f_k(x)=\sum_{1\le i\le k}\mathrm 1_{\scap^p_{\unif_i}}(x)$. Then $f_{k-1}(\unif_{k})$ counts the number of edges from $\unif_{k}$ to $\unif_1,\ldots,\unif_{k-1}$, and $$\setsize{E_{k}(G)}-\setsize{E_{k-1}(G)}=f_{k-1}(\unif_{k}).$$
Similar to the idea in Section~\ref{S.clique upper bound}, we will use Proposition~\ref{P.sym-dom.MGF} here to obtain a bound on the MGF of $\setsize{E(G)}=\sum_{k=1}^nf_{k-1}(\unif_k)$ by a sequence of symmetric rearrangements. We start by defining another sequence of functions $(g_k)_{k \ge 0}$ by setting $g_0 \equiv 0$ on $\sphere$, and iteratively defining 
\begin{equation}\label{E.definition of g}
g_{k} \coloneqq \symmfn{\left(g_{k-1}+\mathrm 1_{\scap^p_{\unif_k}}\right)},\qquad 1\le k\le n-1.
\end{equation}
Using these functions, we define an interpolating process
\begin{equation}m_k \coloneqq \sum_{i=1}^{k}g_{i-1}(\unif_i)+\sum_{i=k+1}^{n}g_k(\unif_i)+\sum_{k+1\le i<j\le n}\indicatorevent{\inner{\unif_i}{\unif_j}\ge\uthreshold{p}}, \ 1 \le k \le n. \label{E.define m_k}
\end{equation}
By definition, we have $m_0 = |E(G)|$ and $m_n = \sum_{k=1}^ng_{k-1}(\unif_k).$ For each $k$, the first term in $m_k$ corresponds to the symmetric rearrangement of the edges between $\unif_1, \ldots, \unif_k$, the second term corresponds to the symmetric rearrangement of the edges between $\unif_1,\ldots,\unif_k$ and $\unif_{k+1}, \ldots, \unif_{n}$, and the final term counts the true edges between $\unif_{k+1},\ldots,\unif_{n}$ in $G$.  

\begin{proposition}\label{P.upper bound for MGF}
    The moment generating functions of $m_k$ are non decreasing in $k$, i.e., for any $\theta\ge 0$, we have $$\bbE \exp(\theta\setsize{E(G)})=\bbE e^{\theta m_0}\le\bbE e^{\theta m_1}\le\ldots\le\bbE e^{\theta m_n}.$$
\end{proposition}

\begin{proof}
    It suffices to show that $\bbE e^{\theta m_{k-1}}\le\bbE e^{\theta m_k}$ for any $1\le k\le n$. Fix $k$, then by first conditioning on $\unif_1,\ldots,\unif_k$, we get
    \begin{align*}
        &\bbE[e^{\theta m_{k-1}}|\unif_1,\ldots,\unif_k]\\
        &=\exp\left(\theta\sum_{i=1}^{k}g_{i-1}(\unif_i)\right)\bbE_{\unif_{k+1},\ldots,\unif_{n}}\exp\left(\theta\left(\sum_{i=k+1}^{n}g_{k-1}(\unif_i)+\sum_{k\le i<j\le n}\indicatorevent{\inner{\unif_i}{\unif_j}\ge \uthreshold{p}}\right)\right)\\
    \end{align*}
    Here $\bbE_{\unif_{k+1},\ldots,\unif_{n}}$ means taking expectation over randomness of $\unif_{k+1},\ldots,\unif_{n}$, we drop the subscript and simply write $\rm E$ to simplify the notation. Let $f(\unif)=g_{k-1}(\unif)+\indicatorfn{\scap^p_{\unif_k}}(\unif)$, and note that $\symmfn{f}=g_k$. By applying Proposition \ref{P.sym-dom.MGF}, we obtain
    \begin{align*}
    &\conditionalE\exp\left(\theta\left(\sum_{i=k+1}^{n}g_{k-1}(\unif_i)+\sum_{k\le i<j\le n}\mathrm 1\{\inner{\unif_i}{\unif_j}\ge\uthreshold{p}\}\right)\right)\\
    =&\conditionalE\exp\left(\theta\left(\sum_{i=k+1}^{n}f(\unif_i)+\sum_{k+1\le i<j\le n}\mathrm 1\{\inner{\unif_i}{\unif_j}\ge\uthreshold{p}\}\right)\right)\\
    \le&\conditionalE\exp\left(\theta\left(\sum_{i=k+1}^{n}\symmfn{f}(\unif_i)+\sum_{k+1\le i<j\le n}\mathrm 1\{\inner{\unif_i}{\unif_j}\ge\uthreshold{p}\}\right)\right)\\
    =&\conditionalE\exp\left(\theta\left(\sum_{i=k+1}^{n}g_{k}(\unif_i)+\sum_{k+1\le i<j\le n}\mathrm 1\{\inner{\unif_i}{\unif_j}\ge\uthreshold{p}\}\right)\right)
    \end{align*}
    Thus $$\bbE[e^{\theta m_{k-1}}|\unif_1,\ldots,\unif_k]\le\bbE[e^{\theta m_{k}}|\unif_1,\ldots,\unif_k].$$
    Taking expectation over $\unif_1,\ldots,\unif_k$ completes the proof.
\end{proof}

We complete the proof of Theorems~\ref{T.SDeviation} and~\ref{T.GDeviation} using the following upper bound on $m_n$.
\begin{proposition}\label{P.upper bound of m_n}
    Let $m_n$ be defined as in \eqref{E.define m_k}. Then 
        \begin{equation}\label{E.upper bound of m_n}
        m_n \leq \left(p + p\deviationrate/2\right)  \binom{n}{2}
        \end{equation}
    holds with probability at least
        \[ 1-\exp \left( -\min \left( c_1(p,\deviationrate)n^2, c_2(p,\deviationrate)n\sqrt{d} \right)\right) .
        \]
    Here $c_1(p,\deviationrate)$ and $c_2(p,\deviationrate)$ are positive constants that depend only on $p$ and $\deviationrate$.
\end{proposition}
We first prove \Cref{T.SDeviation} assuming \Cref{P.upper bound of m_n}, using the following general lemma on stochastic order. The proof of Proposition~\ref{P.upper bound of m_n} is delayed to the end of the section.
\begin{lemma}\label{L.exp dominance}
    Suppose $Y$ and $Y'$  are random variables satisfying $Y,Y'\le 1$, and for all $\theta\ge 0$,
    \begin{equation}
        \bbE e^{\theta Y}\le\bbE e^{\theta Y'}. \label{E.mgf dominance}
    \end{equation}
    Then for all $\delta\ge0$,
    \begin{equation}\label{E.exp dominance}
    \log\bbP(Y\ge\delta)\le\frac{\delta}{2}\log\bbP\left(Y'\ge\frac{\delta}{2}\right)+\log 2.
    \end{equation}
\end{lemma}
\begin{proof}
    Note that for all $\theta\ge0$, 
    $$\bbE e^{\theta Y'}=\bbE\left[e^{\theta Y'}\indicatorevent{Y'\ge\frac{\delta}{2}}\right]+\bbE\left[e^{\theta Y'}\indicatorevent{Y'<\frac{\delta}{2}}\right]\le e^\theta\bbP\left(Y'\ge\frac{\delta}{2}\right)+e^{\theta\frac{\delta}{2}}.$$
    Applying Markov's inequality we obtain
    \begin{align*}
        \log\bbP(Y\ge\delta)&\le\log(e^{-\theta\delta}\bbE e^{\theta Y})\le-\theta\delta+\log\bbE e^{\theta Y'}\le -\frac{\theta\delta}{2}+\log\left(1+\bbP\left(Y'\ge\frac{\delta}{2}\right)e^{\theta-\frac{\theta\delta}{2}})\right).
    \end{align*}
    We get \eqref{E.exp dominance} by taking $\theta=-\log\bbP(Y'\ge\frac{\delta}{2})\ge 0$, since the second term 
    $$\log\left(1+\bbP\left(Y'\ge\frac{\delta}{2}\right)e^{\theta-\frac{\theta\delta}{2}})\right)=\log\left(1+e^{-\frac{\theta\delta}{2}}\right)\le\log 2.$$
\end{proof}
\begin{proof}[Proof of Theorem \ref{T.SDeviation}]
    Let $$Y=\frac{\setsize{E(G)}}{\binom{n}{2}}-p,~~Y'=\frac{m_n}{\binom{n}{2}}-p.$$
    It's clear that both $Y,Y'\le 1$, and they satisfy~\eqref{E.mgf dominance} due to Proposition \ref{P.upper bound for MGF}. Combining with Proposition \ref{P.upper bound of m_n}, we get
    $$\log\bbP(Y\ge\deviationrate)\le\frac{\deviationrate}{2}\log\bbP\left(Y'\ge\frac{\deviationrate}{2}\right)+\log 2\le-\frac{\deviationrate}{2}\min \left( c_1(p,\deviationrate)n^2, c_2(p,\deviationrate)n\sqrt{d} \right)+\log 2.$$
    Note that event $\uedgedeviation{\deviationrate}$ is exactly $Y\ge\deviationrate$, so the proof is complete.
\end{proof}
For the proof of \Cref{P.upper bound of m_n}, it's sufficient to bound each $g_{k-1}(\unif_{k})$ individually, and then sum over $k$. We are going to bound an analogue of the moment generating function and then apply Markov's inequality. Given a measurable function $g:\sphere\to\bbR$, a subset $V\subset\sphere$ and $\lambda\in\bbR$, define 
$$\MGF{g}{V}{\lambda}:=\int_{V}\exp(\lambda g(X))d\sumeasure(X).$$
Let $q=p+p\deviationrate/32$, and select a constant $a>0$ that satisfies 
\begin{equation}
    \label{restriction on a} 
    \uthreshold{p,d}-\frac{a^2}{\sqrt d}>\max\left(\uthreshold{q,d-1},\left(1-\frac{a^2}{\sqrt d}\right)\uthreshold{q,d-1}\right),
\end{equation}
for all $d$ sufficiently large. Such an $a$ exists since Lemma \ref{L.estimate threshold} indicates a gap of order $\Omega(d^{-1/2})$ between $\uthreshold{p,d}$ and $\uthreshold{q,d-1}$. Consider the strip $\strip{t}$, and the complement of a cap $\complementcap{t}$, defined as
\begin{align*}
    \begin{split}
        \strip{t}:=\left\{\bw \in \sphere : |\inner{\bw}{\be}| \le \frac{t}{\sqrt[4]{d}}\right\},\\
        \complementcap{t}:=\left\{\bw \in \sphere : \inner{\bw}{\be} \le \frac{t}{\sqrt[4]{d}} \right\}.
    \end{split}
\end{align*}
The key observation is the following lemma.
\begin{lemma}\label{L.proportion lemma}
    Suppose $\lambda\ge 0$. Then for any function $g$ such that $g=\symmfn{g}$ a.e., and vector $Y\in \strip{a}$,
    $$\MGF{g+\indicatorfn{\scap^p_Y}}{\complementcap{a}}{\lambda}\le\left(1-q+qe^\lambda+2e^{-a^2\sqrt d/2+O(1)+\lambda}\right)\MGF{g}{\complementcap{a}}{\lambda}.$$
\end{lemma}
\begin{proof}
    Let $S_t=\{\unif\in\sphere:\inner{\unif}{\be}=t\}$ be a $d-2$ dimensional slice at $t$ and $\sumeasure_t$ denote the uniform measure on $S_t$. We first prove that for all $|t|\le a/\sqrt[4]{d}$,
    \begin{equation}\label{E.integral estimate}   \int_{S_t}\exp[\lambda\cdot\indicatorfn{\scap^p_{Y}}(\unif)]d\sumeasure_t(\unif)\le1-q+qe^\lambda.
    \end{equation}
    It is equivalent to showing $\sumeasure_t(S_t\cap\scap_Y^p)\le q$, which can be computed directly. Suppose $Y=(y,\sqrt{1-y^2},0,\ldots,0)$, then 
    $$S_t\cap\scap_Y^p=\{\unif=(t,u,\ldots):ty+u\sqrt{1-y^2}\ge\uthreshold{p,d}\}.$$
    After scaling by $1/\sqrt{1-t^2}$, $S_t\cap\scap_Y^p$ is a spherical cap on  $\spherearg{d-2}$ with altitude $(\uthreshold{p,d}-ty)/\sqrt{(1-y^2)(1-t^2)}$, and so
    %The marginal of $\sumeasure_t$ on $u$ is proportional to the 2 dimensional marginal of $\sumeasure$ restricted on first coordinate is $t$. Thus we can write
    \begin{align*}
        \sumeasure_t(S_t\cap\scap_Y^p)=\utaild{d-1}{\frac{\uthreshold{p,d}-ty}{\sqrt{(1-y^2)(1-t^2)}}}.
    \end{align*}
    Recall that $|t|,|y|\le a/\sqrt[4]{d}$ and $a$ satisfies \eqref{restriction on a}. This is upper bounded by $\utaild{d-1}{\uthreshold{q,d-1}}=q$, which proves \eqref{E.integral estimate}.
    
    Next, we show that the integral on $\complementcap{a}-\strip{a}$ is small. For those $t<-a/\sqrt[4]{d}$, the integral on $\complementcap{a}-\strip{a}$ can be bounded by that on $\strip{a}$ since $g$ is non-decreasing in the direction $\eb$ as
    \begin{align*}
        &\MGF{g+\indicatorfn{\scap^p_Y}}{\complementcap{a}-\strip{a}}{\lambda}
        \le e^\lambda\MGF{g}{\complementcap{a}-\strip{a}}{\lambda}\\
        \le& e^\lambda \frac{\sumeasure(\complementcap{a}-\strip{a})}{\sumeasure(\complementcap{a})}\MGF{g}{\complementcap{a}}{\lambda}\le 2e^{-a^2\sqrt d/2+O(1)+\lambda}\MGF{g}{\complementcap{a}}{\lambda}.
    \end{align*}
    Here in the last inequality, we apply Lemma \ref{sodin's Lemma} on $\sumeasure(\complementcap{a}-\strip{a})$. Multiplying \eqref{E.integral estimate} by the value of $g$ on $S_t$ and integrating over $|t|\le a/\sqrt[4]{d}$ gives
    \begin{align*}
    \MGF{g+\indicatorfn{\scap^p_Y}}{\complementcap{a}}{\lambda}&=\MGF{g+\indicatorfn{\scap^p_Y}}{\strip{a}}{\lambda}+\MGF{g+\indicatorfn{\scap^p_Y}}{\complementcap{a}-\strip{a}}{\lambda}\\
    &\le\left(1-q+qe^\lambda+2e^{-a^2\sqrt d/2+O(1)+\lambda}\right)\MGF{g}{\complementcap{a}}{\lambda}.
    \end{align*}
\end{proof}
Finally, we derive an upper bound on $g_{k-1}(\unif_{k})$ from Lemma \ref{L.proportion lemma} and complete the proof of Proposition \ref{P.upper bound of m_n}.
\begin{proof}[Proof of \Cref{P.upper bound of m_n}]
Let $\lambda$ and $c_0$ be constants to be determined later and define
$$h\coloneqq \min\left(\frac{c_0\sqrt n}{\sqrt[4]{d}},\frac{a}{2}\right).$$  Define the index set 
$$I:=\{1\le k\le n:\unif_k\notin \strip{h}\}.$$
These vectors $\{\unif_k\}_{k\in I}$ are atypical and have large contributions to $\MGF{g_k}{\complementcap{a}}{\lambda}$. Consider the bad event $\mathcal B:=\{\setsize{I}\ge p\deviationrate n/16\}$ of having too many atypical vectors. Its probability can be controlled as follows. By Lemma \ref{sodin's Lemma}, for a single $k$,
$$\bbP(\unif_k\notin\strip{h})=2\utail{h/\sqrt[4]{d}}\le\exp\left(-\sqrt{d}h^2/2+O(1)\right).$$
Then an union bound gives
$$\bbP(\mathcal B)\le 2^n\cdot\bbP\left(\unif_k\notin\strip{h},\forall 1\le k\le  \frac{p\deviationrate n}{16}\right)\le2^n\exp\left(\frac{p\deviationrate n}{16}\left(-\sqrt{d}h^2/2+O(1)\right)\right)\le \exp\left(-p\deviationrate n\sqrt{d}h^2/64\right).$$
Since $g_k = \symmfn{\left(g_{k-1}+\indicatorfn{\scap^p_{\unif_k}}\right)}$, we can use Lemma~\ref{L.proportion lemma} to bound $\MGF{g_k}{\complementcap{a}}{\lambda}$ as
\begin{equation*}
        \MGF{g_k}{\complementcap{a}}{\lambda}\le\MGF{g_{k-1}+\indicatorfn{\scap^p_{\unif_k}}}{\complementcap{a}}{\lambda}\le
        \begin{cases}
             \left(1-q+qe^\lambda+o(1)\right)\MGF{g_{k-1}}{\complementcap{a}}{\lambda}& \text{$k\notin I$,}\\
            e^\lambda\MGF{g_{k-1}}{\complementcap{a}}{\lambda}& \text{$k\in I$.}\\
        \end{cases}
    \end{equation*}
    Note that $\MGF{g_0}{\complementcap{a}}{\lambda}=\sumeasure(\complementcap{a})\le 1$, we can deduce from above that for all $1\le k\le n$
    $$ \MGF{g_k}{\complementcap{a}}{\lambda}\le(1-q+qe^\lambda+o(1))^k\exp\left(\lambda\setsize{I\cap\{1,\ldots,k\}}\right).$$
    In particular, on $\mathcal B^c$, $\setsize{I\cap\{1,\ldots,k\}}\le p\deviationrate n/16$ and thus for all $1\le k\le n$,
    $$\log \MGF{g_k}{\complementcap{a}}{\lambda}\le k\log(1-q+qe^\lambda+o(1))+\lambda p\deviationrate n/16.$$
    From here we derive a bound on $g_{k-1}(\unif_k)$ by using Markov's inequality. Since $h \leq a/2 = O(1)$, we have $$\sumeasure(\complementcap{a}-\complementcap{h})\asymp \sumeasure(\complementcap{h}^c)=\exp\left(-\sqrt{d}h^2/2+O(1)\right).$$
    Then for all $k\notin I$, $\unif_k\in\strip{h}$ and
    $$\MGF{g_{k-1}}{\complementcap{a}}{\lambda}\ge \MGF{g_{k-1}}{\complementcap{a}-\complementcap{h}}{\lambda}\ge\sumeasure(\complementcap{a}-\complementcap{h})e^{\lambda g_{k-1}(\unif_k)},$$
    which implies
    \begin{align*}
    g_{k-1}(\unif_k)&\le\frac{1}{\lambda}\log\MGF{g_{k-1}}{\complementcap{a}}{\lambda}-\frac{1}{\lambda}\log\sumeasure(\complementcap{a}-\complementcap{h}) \\
    &\le\frac{\log(1-q+qe^\lambda+o(1))}{\lambda} (k-1)+p\deviationrate n/16+\frac{\sqrt{d}h^2}{2\lambda}+O(1).
    \end{align*}
    For $k\in I$, $\unif_k\notin \strip{h}$, and $g_{k-1}(\unif_k)\le n$ is trivially true.
    Finally, we upper-bound $m_n$. Summing over $k$, we get that on $\mathcal B^c$,
    \begin{align*}
        m_n&=\sum_{k\notin I}g_{k-1}(\unif_k)+\sum_{k\in I}g_{k-1}(\unif_k)\\
        &\le \frac{\log(1-q+qe^\lambda+o(1))}{\lambda}\binom{n}{2}+p\deviationrate n^2/16+\frac{n\sqrt{d}h^2}{2\lambda}+O(n)+p\deviationrate n^2/16.
    \end{align*}
    Choose $\lambda$ such that $\frac{1}{\lambda}\log(1-q+qe^\lambda)\le p+p\deviationrate/16$, and $c_0\le\sqrt{p\deviationrate\lambda}/10$. Such a $\lambda$ exists since     $$\underset{\lambda\to 0^+}{\lim}\frac{\log(1-q+qe^\lambda)}{\lambda}=q<p+p\deviationrate/16.$$
    Then $m_n$ is bounded by 
    $$(p+p\deviationrate/16)\binom{n}{2}+p\deviationrate n^2/16+p\deviationrate n^2/16+O\left(\frac{n}{\lambda}\right)+p\deviationrate n^2/16\le \left(p+p\deviationrate/2\right)\binom{n}{2}.$$
    So we can conclude that \eqref{E.upper bound of m_n} holds on $\mathcal B^c$. The proof of Proposition \ref{P.upper bound of m_n} is completed by noting that
    $$\bbP(\mathcal B)\le \exp(-p\deviationrate n\sqrt{d}h^2/64)\le\exp(-\min(c_1(p,\deviationrate)n^2, c_2(p,\deviationrate)n\sqrt d)),$$
    where $c_1(p,\deviationrate)=-p\deviationrate c_0^2/64$ and $c_2(p,\deviationrate)=-p\deviationrate a^2/64$ are constants.
\end{proof}
\begin{remark}\label{R.explicit constants}
    We can determine the dependence of the constants in Theorem \ref{T.GDeviation} (and also Theorem \ref{T.SDeviation}) on $p$ and $\varepsilon$ by estimating the constants used in the proof of Proposition \ref{P.upper bound of m_n}. Take $\lambda=\deviationrate/100$ and $c_0=\sqrt{p}\deviationrate/100$. Note that $\gtailinverse{p}\sim\sqrt{\log 1/p}$, so we may take
    $$a^2=\frac{\gtailinverse{p}-\gtailinverse{p+p\deviationrate/32}}{2}=\Omega\left(\frac{\deviationrate}{\sqrt{\log(1/p)}}\right).$$
    Then the corresponding constants satisfy $c_1(p,\deviationrate)=\Omega(p^2\deviationrate^3)$ and $c_2(p,\deviationrate)=\Omega(p\deviationrate^2/\sqrt{\log 1/p})$. This gives $$\bbP(\uedgedeviation{\deviationrate})\le\exp\left(-\min\left(\Omega(p^2\deviationrate^4)n^2,\Omega\left(\frac{p\deviationrate^3}{\sqrt{\log 1/p}}\right)n\sqrt{d}\right)\right).$$
\end{remark}

\section*{Acknowledgements}

We would like to thank Jian Ding for introducing the problem, and for valuable discussions and suggestions during the early stages of this project. This work was supported by the National Key R\&D program of China (No. 2023YFA1010103).

%\printbibliography
\bibliographystyle{plainnat}
\bibliography{bibliography}

\end{document}